\documentclass[12pt]{article}

\usepackage{amssymb,amsmath,amsthm}
\usepackage{graphicx}
\usepackage{geometry}
\usepackage{natbib}
\usepackage{hyperref}

\newtheorem{theorem}{Theorem}
\newtheorem{corollary}{Corollary}

\newtheorem{lemma}{Lemma}

\theoremstyle{definition}
\newtheorem{example}{Example}
\newtheorem{examplecont}{Example}

\newcommand{\Exp}[2]{{\text{E}_{#1}}[ \ensuremath{ #2 } ]  }

\title{Bayes-optimal prediction with frequentist coverage control}
\author{Peter Hoff \\ Department of Statistical Science \\ Duke University}
\date{\today}

\begin{document}
\maketitle

\begin{abstract}
This article illustrates how indirect or prior information can be 
optimally used to construct a 
prediction region that maintains a target frequentist coverage rate.
If the indirect information is accurate, the volume of the prediction region is lower on average than that of other regions with the same coverage rate. Even if the indirect information is
inaccurate, the resulting region still maintains the target coverage rate. 
Such a prediction region can be constructed for models 
that have a complete sufficient statistic, which includes 
many widely-used parametric and nonparametric models. 
Particular examples include a
Bayes-optimal conformal prediction procedure that maintains a constant coverage 
rate across distributions in a nonparametric model, as well as a prediction 
procedure for the normal linear regression model 
that can utilize a regularizing prior distribution, 
yet  maintain 
a frequentist coverage rate that is constant as a function of 
the model parameters and explanatory variables. 
No results in this article rely on asymptotic approximations.

\smallskip


\end{abstract}

\section{Introduction} 
A standard statistical inference task is to construct a prediction region, that is,  a set of plausible values for an unobserved random object $Y$ having sample space $\mathcal Y$ based on a realization of another
 random object $X$ having sample space $\mathcal X$.   
A non-randomized procedure for constructing 
a prediction region 
is a set-valued function 
$\mathcal X \to  2^\mathcal Y$, for example 
denoted by $x\mapsto A_x$, where
$A_x\subset\mathcal Y$ is the set of 
predicted values for $Y$ when $X$ is observed to be $x$. 

A prediction procedure is precise if its expected volume 
is small, and it is accurate if
its coverage probability is high. 
The coverage probability of a prediction procedure $x\mapsto A_x$ 
is typically defined
as the probability of the event $Y\in A_X$.
More generally, 
let $\mathcal P = \{ P_\theta : \theta\in \Theta\}$ be a
collection of joint probability distributions on 
measurable subsets of 
$\mathcal X\times \mathcal Y$.  
The coverage probability of the set-valued function $x\mapsto A_x$
for a given $\theta$ is simply $P_\theta(A)$, 
where $A=\{ (x,y) : y\in A_x\}$ is the graph of
 $x\mapsto A_x$.
A prediction procedure $A$ for which 
$P_\theta(A)\geq 1-\alpha$ for some target value of $\alpha$ and for all $\theta$ is generally referred to as having 
$1-\alpha$ frequentist coverage
\citep{aitchison_dunsmore_1975}. 
In what follows, if
$P_\theta(A)= 1-\alpha$ for all $\theta$ then $A$ will be referred to 
as a $1-\alpha$ constant coverage prediction region. 
\citet{patel_1989} reviews a variety of 
methods for constructing prediction regions 
that have frequentist coverage control, the most widely-used among them perhaps being those that are derived from pivotal 
quantities \citep{lawless_fredette_2005}. 

The precision of a prediction region can be quantified 
with its expected volume.
Let $\mu$ be a volume measure on 
measurable subsets of $\mathcal Y$. 
We define the risk function of
$A$ to be 
the  expected volume of $A_X$  as a function of $\theta$: 
\begin{equation}
R_\theta(A)  = {\rm E}_{\theta}[ \mu(A_X) ].   
\label{eqn:risk}
\end{equation} 
As with confidence regions, 
precision can be increased at the expense of coverage, and so we
compare among, or optimize over, regions that share a common 
coverage. 
A prediction region $A$ is better than $A'$ if  
$R_\theta(A)\leq R_\theta(A')$  and 
$P_\theta(A) \geq P_\theta(A')$ for all $\theta$, with inequality 
for some $\theta$. 
However, as with other types of statistical decision problems, 
typically there is 
not a uniformly best prediction region procedure. 
This motivates the identification of 
an optimal member of a reduced class of procedures, 
or identification of a procedure that performs well in a 
particular region of the parameter space. 
As an example of the former approach, \citet{evans_fraser_1980} 
show that for the normal linear model, the standard prediction region
obtained by pivoting is,
for a given coverage rate, 
optimal among all regions that are equivariant with respect to 
affine transformations. Such a prediction 
region performs equally well across the parameter space, 
in the sense that the expected volume of the prediction region 
(appropriately scaled by the variance) 
does not 
depend on the model parameters. 

In this article, prediction regions that share a 
common frequentist coverage are 
compared to each other in terms of their 
average risk across the parameter space. 
The motivation for this criterion is that 
in many applications there is indirect information or prior knowledge
that some distributions in $\mathcal P$ are more plausible than 
others. In such cases, it may be preferable to have 
a prediction region that performs  well for 
values of $\theta$ that are most plausible, at the expense 
of worse performance for less plausible values. 
This suggests evaluating a prediction region $A$ with a
Bayes risk,
\begin{equation}
 R(A) = \int R_\theta(A) \, \pi(d\theta), 
\label{eqn:brisk}
\end{equation} 
where $\pi$ is a probability distribution that 
gives a large weight to $\theta$-values for which 
$R_\theta(A)$ is desired to be small.  For example, 
 $\theta$ could be a vector 
of linear regression coefficients, many of which are suspected to 
be close or equal to zero. 
In this case, 
we might prefer a prediction region that has particularly 
low expected volume when $\theta$ is sparse, over, for example,
an equivariant region
that has the same expected volume for all $\theta$-values. 
While we will refer to $\pi$ as a prior distribution and 
$R(A)$ as a Bayes risk,
$\pi$ may or may not be a formal prior distribution
that describes subjective beliefs about the value of
$\theta$. Most generally, $\pi$ can simply be viewed as a weighting function 
that prioritizes regions of the parameter space.

The main result of this article is that, 
for many commonly-used statistical models, 
it is possible to find a prediction region that 
maintains a target frequentist coverage rate and is also 
Bayes-optimal for 
its coverage. 
Specifically, 
we are able to construct a set  $A^\pi \subset \mathcal X\times \mathcal Y$
such that 
$P_\theta(A^\pi) \geq 1- \alpha$ for all  $\theta$, and 
$R(A^\pi) \leq R(A)$ for all sets
 $A$  having the same coverage as $A^\pi$. 
In particular, if $A^\pi$ has a constant coverage rate of $1-\alpha$, 
then 
$A^\pi$ is Bayes-optimal among all $1-\alpha$ 
constant coverage 
prediction regions. 
Such a prediction region $A^\pi$ is 
Bayes-optimal among procedures 
that share its
frequentist coverage rate. We 
follow \citet{yu_hoff_2018} by 
referring to such a procedure 
as ``frequentist and Bayesian'', or FAB. 
Statistical procedures of this type  go back at least to 
\citet{pratt_1963}, who constructed 
a constant coverage confidence interval for the mean of 
a normal population that has minimum prior expected width among 
intervals with the same frequentist coverage. 
Related to 
this is the ``Bayes-non-Bayes compromise'' of \citet{good_1992}, 
whereby test statistics are obtained from Bayes factors, but are compared 
to frequentist null distributions, typically via permutation. 
\citet{yu_hoff_2018} and \citet{hoff_yu_2019} extended 
Pratt's Bayes-optimal frequentist interval 
to multiparameter settings, in which the prior distribution may be 
empirically estimated from the data using a hierarchical model, 
resulting in an adaptive 
confidence interval procedure with guaranteed frequentist coverage, even if the hierarchical model is wrong.

This article shows how to construct analogous
procedures for prediction in models
that have a complete sufficient statistic. 
This includes many 
exponential family models such as multinomial, linear regression and 
generalized linear 
models, some models with varying support, nonparametric models 
\citep{bell_blackwell_breiman_1960}, and others \citep{mattner_1996}.  
As shown in \citet{faulkenberry_1973}, when a sufficient statistic is 
available, a prediction region with a desired coverage rate may be constructed
by inverting the acceptance regions 
of a collection of conditional point-null hypothesis tests. 
Faulkenberry also observed that prediction 
regions with constant coverage must have constant 
conditional coverage, given a complete sufficient statistic. 
In the next section
we show how \emph{all}
prediction regions may be constructed from
conditional tests, and how all
prediction regions with common 
coverage as a function of the parameter 
must also have common conditional coverage as a function 
of the complete sufficient statistic. 
This fact is used in Section 3 to show how a Bayes-optimal prediction 
region for a given coverage  may be found by constructing 
Bayes-optimal conditional tests. 
This is done by obtaining a joint disintegration of the
the 
probability 
measures $\{ P_\theta:\theta\in\Theta\}$ that define coverage, and the
non-probability measure $R$ that defines the Bayes risk. 
The conditional tests that 
minimize the risk 
subject to the constraint on frequentist coverage can then be identified with 
a variant of the Neyman-Pearson lemma. 
None of these results rely on asymptotic approximations.

While the methodology in this article applies generally 
to models with  a complete sufficient statistic, 
Section 4 considers a few specific scenarios in detail, 
including prediction for multivariate normal and normal linear regression 
models, as well as nonparametric prediction using conformity scores 
as described in \citet{gammerman_vovk_vapnik_1998}.
In particular, it is shown that the Bayes-optimal 
choice of a conformity score is,
not surprisingly, 
the Bayesian posterior predictive density. 
The resulting FAB prediction region is Bayes optimal among 
nonparametric regions with constant frequentist coverage, 
and is of course different from a fully Bayesian posterior predictive region, 
as the latter does not have constant frequentist coverage. 
In this sense, the conformal FAB approach optimally   
incorporates prior information while maintaining a constant frequentist coverage rate. 
Additional aspects of the FAB prediction  methodology, 
and some directions for further research are 
discussed in Section 5. Proofs are in an appendix.

\section{Prediction regions via sufficiency}

\subsection{Review of Faulkenberry's construction} 
Let $X$ and $Y$ be random objects taking values in spaces $\mathcal X$ and 
$\mathcal Y$ respectively, with a 
joint distribution that 
is a member of a model $\mathcal P=\{ P_\theta: \theta\in \Theta\}$ 
on $\mathcal X\times \mathcal Y$. 
For a model $\mathcal P$ with a sufficient statistic $Z:\mathcal X\times \mathcal Y \rightarrow  \mathcal Z$, 
\citet{faulkenberry_1973} proposed constructing a $1-\alpha$ 
prediction 
region $A$ for $Y$ from $X$ as follows: 
\begin{enumerate}
\item For each $z\in \mathcal Z$, 
\begin{enumerate} 
\item identify $P^Y_z$, the conditional distribution of $Y$ given $\{ Z=z \}$,
      which by sufficiency does not depend on $\theta$; 
\item find a set $C_z\subset \mathcal Y$ such that
     $P_z^Y( C_z ) = 1-\alpha$.  
\end{enumerate}
\item Let  $A = \{ (x,y): y\in C_{Z(x,y)} \}$,
      so $y\in A_x \Leftrightarrow y\in C_{Z(x,y)}$.
\end{enumerate}
The conditional probability of the event $Y\in A_X$ 
given $\{Z=z\}$ is $P_z^Y ( C_z) = 1-\alpha$ for all $z$, and so its unconditional probability is equal to 
$1-\alpha$ as well.

\begin{example} 
\label{ex:simplenormal}
Let $X\sim N(\theta,k\sigma^2)$ and 
$Y\sim N(\theta,\sigma^2)$ be independent
for some unknown value of $\theta\in \mathbb R$ and known value of 
$\sigma^2>0$.  Then $Z= (X + kY)/(1+k)$ is a complete sufficient 
statistic, with $Y|\{Z=z\} \sim N(z,\sigma^2/(k+1))$. 
For any function $\delta:\mathbb R\rightarrow\mathbb R$, 
define for each $z\in \mathbb R$ 
the set  
\begin{equation}
 C_z = \left \{ y: \left | \frac{z-y}{ \sigma/\sqrt{k+1}} + \delta(z) 
 \right| ^2 
 < \chi^2_{1, \delta(z)^2,1-\alpha} \right \} ,
\end{equation}
where $\chi^2_{1,\delta^2,1-\alpha}$ is the $1-\alpha$ 
quantile of the $\chi^2_1$ distribution with 
noncentrality parameter $\delta^2$. 
Then $P_z^Y(C_z)= 1-\alpha$ for each $z$.
Defining $A$ as in step 2 gives
\begin{equation}
 A = \left \{ (x,y): \left | \frac{x-y}{ \sigma\sqrt{k+1}} + \delta(Z(x,y)) \right| ^2 
 < \chi^2_{1, \delta(Z(x,y))^2,1-\alpha} \right \},
\end{equation}
which is a $1-\alpha$ constant coverage prediction region 
for any choice of $\delta$. For example, if 
$\delta$ is identically zero then 
$A_x$  is  the 
standard prediction 
interval $x\pm \sigma\sqrt{k+1}\times \Phi^{-1}(1-\alpha/2)$ obtained via the pivotal quantity $X-Y$, where 
$\Phi^{-1}$ is the standard normal quantile function. 
In Section 4 we will obtain the 
function $\delta$ that minimizes the Bayes risk 
under a normal prior distribution for $\theta$, 
and generalize the method to multivariate normal models 
with unknown variance. 

\end{example}

Faulkenberry specifically  considered 
the case that $X$ and $Y$ are independent and 
that $Z$ is a complete sufficient statistic, 
but only sufficiency - not independence or completeness - is 
necessary to ensure that a set obtained from Faulkenberry's construction has  $1-\alpha$ constant coverage.  Faulkenberry considers
completeness because it provides a
characterization of the set of prediction 
regions that have $1-\alpha$ constant coverage, that is, 
regions for which $P_\theta(A) = 1-\alpha$ for all $\theta$. 
For such a region, Faulkenberry points out that
$P_\theta( A | Z ) = 1-\alpha$ almost surely for each $\theta$, 
because  ${\rm E}_\theta[ P_\theta( A | Z ) ] = 
  P_\theta(A) =1-\alpha$ for all $\theta$
and $Z$ is complete.
Faulkenberry therefore concludes that, for models with a complete
sufficient statistic,
a prediction region has 
constant coverage as a function of the parameter 
if and only if it has constant conditional coverage 
(almost surely) as a function of the 
complete sufficient statistic.

Can all prediction procedures for models with a complete sufficient statistic be derived from Faulkenberry's construction? The answer is no, for two reasons:
The first reason is that not all subsets of $\mathcal X\times \mathcal Y$
can be expressed 
as $A = \{ (x,y): y\in C_{Z(x,y)} \}$ for some set-valued function $C$ 
mapping $\mathcal Z$ to subsets of $\mathcal Y$, unless 
some additional conditions on $Z$ are met. 
The second reason is that, as pointed out by \citet{dunsmore_1976}, there may not exist
prediction procedures $A$ with constant coverage, for example, 
if the $P_\theta$'s are discrete distributions. 
In this case, coverage above or equal to $1-\alpha$ may still be maintained
by choosing a $C_z$ in Faulkenberry's construction so 
that $P_z^Y(C_z)\geq 1-\alpha$ for all $z$.
However, while every set with conditional  
coverage  of at least $1-\alpha$ also has marginal coverage 
of at least $1-\alpha$, the converse is not necessarily true, even with 
completeness of $Z$: 
For prediction of a binomial random variable,  Dunsmore 
provides an example of a 
prediction region  $A$ 
for which $P_\theta( A ) \geq 1-\alpha$ for all $\theta$
but  $P( A |Z=z) $ falls below $1-\alpha$ for some values of $z$. 
This means that in some cases 
the set of prediction regions having 
coverage greater than $1-\alpha$ for all $\theta$ 
is a proper superset of 
those with conditional coverage greater than  $1-\alpha$ for all $z$, and so 
the best procedure 
that can be obtained from Faulkenberry's construction may not be the best 
procedure obtainable. 

In the remainder of this section we extend Faulkenberry's method to address these limitations.
In the next subsection we generalize Faulkenberry's construction
to provide a complete characterization of
set-valued functions from   $\mathcal X$ to subsets of
$\mathcal Y$ in terms of set-valued functions from
$\mathcal Z$ to subsets of
$\mathcal X\times \mathcal Y$.
In Subsection 2.3 we show how complete sufficiency 
allows for a characterization of all procedures that have the same coverage, in terms of collections of point-null hypothesis tests that have the same conditional size.  
In particular, given a candidate prediction procedure, we 
can characterize the 
class of procedures with the same coverage, 
and possibly find one with lower risk. 

\subsection{Characterizing set-valued functions}
Let $\mathcal X$ and $\mathcal Y$ be spaces. 
For a subset $A$ of $\mathcal X\times \mathcal Y$ and 
element $x\in \mathcal X$, 
the \emph{section} of $A$ at $x$ is the 
set  $A_x = \{ y: (x,y) \in A\}$, a subset of $\mathcal Y$.
The sections of $A$ define a set-valued function 
$\mathcal X \rightarrow 2^\mathcal Y$ given by 
$x\mapsto A_x$.  
Conversely, 
for each $x\in  \mathcal X$, let $A_x$ be a subset of $\mathcal Y$. 
Every such set-valued function has a \emph{graph} $A = \{ (x,y) : y\in A_x \}$, a subset of 
$\mathcal X\times \mathcal Y$.
The operations of calculating the graph of a set-valued function 
$\mathcal X \rightarrow 2^\mathcal Y$, 
and  
calculating the sections of a subset of $\mathcal X\times \mathcal Y$, 
are inverses of each other, 
and so there is a bijection between set-valued functions from $\mathcal X \rightarrow 2^\mathcal Y$ and subsets of $\mathcal X\times \mathcal Y$.  
As such, in what follows we will use
the same symbol (e.g.\ ``$A$'') for a
set-valued function and its graph, 
and whether or 
not the symbol represents a function $\mathcal X\rightarrow 2^\mathcal Y$ 
or a subset of $\mathcal X\times \mathcal Y$  will be clear from context. 

Another representation of a subset of $\mathcal X\times \mathcal Y$ is 
given by any 
surjective mapping $Z:\mathcal X\times \mathcal Y \rightarrow \mathcal Z$. Overloading notation somewhat, 
for each $z\in \mathcal Z$ let $A_z = A\cap Z^{-1}\{z\}$. 
Then the mapping $z\mapsto A_z$ 
is a set-valued function 
from $\mathcal Z$ to $2^{\mathcal X\times \mathcal Y}$.  
Conversely, if 
$A_z$ is a subset of the fiber $Z^{-1}\{z\}$ for each $z\in \mathcal Z$, 
then $\cup_{z\in \mathcal Z} A_z$ is a subset of $\mathcal X\times \mathcal Y$. 
We summarize this observations and those of the preceding paragraph as follows:
\begin{lemma}
\label{lem:svfr} 
Let $Z: \mathcal X \times \mathcal Y \rightarrow \mathcal Z$ be a surjection. 
Then there is a bijection between each pair of the following sets:
\begin{enumerate}
\item $2^{\mathcal X\times \mathcal Y}$; 
\item set-valued functions $\mathcal X \rightarrow 2^\mathcal Y$;  
\item set-valued functions $A: \mathcal Z \rightarrow \mathcal X\times \mathcal Y$ for which $A_z \subset Z^{-1}\{z\}$.  
\end{enumerate}
\end{lemma}
In the next subsection, we will show how the bijection between 
items 2 and 3 can be used to represent a 
prediction procedure (item 2) as a collection of acceptance regions of 
hypothesis tests (item 3).

Now recall that Faulkenberry's method is to construct 
a prediction procedure $A:\mathcal X\rightarrow 2^\mathcal Y$ from 
a set-valued function $C:\mathcal Z \rightarrow 2^\mathcal Y$
by setting $A = \{ (x,y):  y \in C_{Z(x,y)} \}$. 
Can all 
prediction procedures 
be represented in this way? 
Since Lemma \ref{lem:svfr} shows there is a bijection between 
functions $\mathcal X \rightarrow 2^\mathcal Y$ and functions 
$\mathcal Z \rightarrow  2^{\mathcal X\times \mathcal Y}$, intuitively 
Faulkenberry's representation can only be complete if 
somehow $y$ and $Z(x,y)$ determine $x$, for example if 
$Z(x,y)=x+y$. In fact, such a condition is
necessary and sufficient for Faulkenberry's representation to be a
complete  characterization of the  
set-valued functions from $\mathcal X$ to $2^{\mathcal Y}$, 
or equivalently, subsets of $\mathcal X\times \mathcal Y$:
\begin{lemma}    
\label{lem:injectiverep}
Let $Z:\mathcal X\times \mathcal Y \rightarrow \mathcal Z$. 
A set  $A\subset \mathcal X\times \mathcal Y$ 
can be written as 
$A = \{ (x,y): y\in C_{Z(x,y)} \}$ for some $C:\mathcal Z \rightarrow 2^\mathcal Y$ 
if and only if 
the function $\mathcal X\rightarrow \mathcal Z$ given by $Z_y(x)= Z(x,y)$ is injective for each $y$. 
\end{lemma}

Similarly, if $x$ and $Z(x,y)$ determine $y$, then the subsets of $\mathcal X\times \mathcal Y$ can 
be characterized in terms of set-valued functions from $\mathcal Z$ to 
$2^\mathcal X$:
\begin{corollary}
\label{cor:accF} 
A set  $A\subset \mathcal X\times \mathcal Y$
can be written as
$A = \{ (x,y): x\in B_{Z(x,y)}\}$ for some 
$B:\mathcal Z \rightarrow 2^\mathcal X$
if and only if
the function $\mathcal Y\rightarrow \mathcal Z$ given by $Z_x(y)= Z(x,y)$ is injective for each $x$.
\end{corollary} 

\begin{examplecont}[Continued]
For the statistic $Z(x,y)=(x+ky)/(1+k)$, both
$Z_x$ and $Z_y$ are injective for each $x$ and $y$. 
Let  $A$ be a subset of $\mathcal X\times \mathcal Y$, 
and let 
$C_z = \{ y: (z\times (1+k)-ky,y) \in  A  \} $ and 
$B_z = \{ x: (x, (z\times(1+k)-x)/k) \in  A  \}$. Then 
$A = \{ (x,y) : y \in C_{Z(x,y)} \} 
   = \{ (x,y) : x \in B_{Z(x,y)} \}$. 
\end{examplecont}

\begin{example}
Let $X\sim N(\theta,1)$ and 
$Y|\{X=x\}\sim N(\theta+ x,1)$ with unknown $\theta\in \mathbb R$. 
Then $Z(x,y) = y$ is a complete sufficient statistic. 
In this case, $Z_y$ is not injective, but $Z_x$ is, since $Z_x(y) = y$. 
Faulkenberry's construction does not characterize the prediction regions 
in this case, 
but a modification  does: For any $A\subset \mathcal X\times \mathcal Y$, let 
$B:\mathcal Y \rightarrow 2^{\mathcal X}$ be defined as 
$B_y = \{ x: (x,y) \in  A\}$, the section of the reflection of 
$A$ at $y$. Then clearly 
$A = \{ (x,y) : x \in B_y \}$. 
A $1-\alpha$ prediction region $A$ for $Y$ can be constructed 
by choosing $B_y$ so that $P_y^X( B_y ) =1-\alpha$ for each $y$ 
and then letting $A = \{ (x,y): x\in B_y \}$. 
\end{example} 

\begin{example} 
Let  $X \sim N_p(0, \sigma^2 I)$, 
and let $Y|\{ X= x\} \sim N_p( \theta x ,  I)$, where $\theta\in \mathbb R$ and $\sigma^2 \in \mathbb R^+$ are both unknown. 
Then 
$Z( x,  y) = 
  ( \lVert x \rVert  , x^\top y)$ is a complete sufficient statistic, 
but  $Z_{y}$ is not injective for any $y$, nor is 
$Z_{x}$ injective for any $x$. 
A subset $A$ of $\mathcal X\times \mathcal Y$ cannot in general 
be represented as 
$\{ (x,y) : x \in B_{Z(x,y)} \}$ or 
$\{ (x,y) : y  \in C_{Z(x,y)} \}$ 
for set-valued functions  
$B:\mathcal Z \rightarrow 2^\mathcal X$ or 
$C:\mathcal Z \rightarrow 2^\mathcal Y$. 
However, any $A\subset \mathcal X\times \mathcal Y$ can be expressed as $A = \{ (x,y) : (x,y) \in A_{Z(x,y)}\}$ 
simply by setting
$A_z = A \cap Z^{-1}\{z\}$. 
\end{example}

To summarize, a prediction region is a set-valued function
$\mathcal X\rightarrow 2^\mathcal Y$, or equivalently, 
a subset of $\mathcal X \times \mathcal Y$. 
For any surjective function $Z:\mathcal X\times \mathcal Y \rightarrow \mathcal Z$, 
there is a bijection between such 
set-valued functions and
set-valued functions $A:\mathcal Z \rightarrow 2^{\mathcal X \times \mathcal Y}$ for which $A_z \subset Z^{-1}\{z\}$ for all $z$ in $\mathcal Z$. 
Furthermore, if the function $y\mapsto  Z(x,y)$ is injective for each $x$, 
then there is a bijection between prediction regions and 
functions that map   $\mathcal Z$ to  $2^\mathcal X$. 
Similarly, if the function $x \mapsto Z(x,y)$ is injective for each $y$, 
then 
there is a bijection between prediction regions and 
functions that map $\mathcal Z$ to $2^\mathcal Y$.

\subsection{Characterizing regions as  hypothesis tests}
Let
$(\mathcal X, \mathcal F)$ and
$(\mathcal Y, \mathcal G)$
be measurable spaces, and let 
$(\mathcal X \times \mathcal Y,\mathcal A )$ be their product space, 
so that  $\mathcal A= \mathcal F\otimes \mathcal G$ is 
the smallest $\sigma$-algebra containing the rectangles $\{ F\times G: F\in \mathcal F, G\in \mathcal G\}$. 
As described previously, 
a prediction procedure for a $\mathcal Y$-valued random object $Y$ from 
an $\mathcal X$-valued random object $X$
is a set-valued function
$A: \mathcal X \rightarrow 2^{\mathcal Y}$, 
or alternatively, a subset $A$  of $\mathcal X\times \mathcal Y$. 
The
coverage of $A$ 
under any probability measure $P$ on $(\mathcal X\times \mathcal Y,\mathcal A)$ 
is simply $P(A)$, 
which in order to be well-defined, 
requires  $A$ to be $\mathcal A$-measurable. 
Conversely, if $A\in \mathcal A$ 
then $A_x \equiv\{ y:(x,y) \in A\} $ is a measurable subset of 
$\mathcal Y$ for each $x\in \mathcal X$, that is, $A:\mathcal X \rightarrow \mathcal G$ 
\citep[Theorem 34.A]{halmos_1950}. 
However, not all set-valued functions
$\mathcal X \rightarrow \mathcal G$ have $\mathcal A$-measurable graphs. 
\citep[Chapter 14]{rockafellar_wets_1998}.

Let $\mathcal P = \{ P_\theta : \theta\in \Theta\}$ 
be a statistical model for $(\mathcal X\times \mathcal Y, \mathcal A)$ that has
a boundedly complete regular sufficient statistic
$Z:(\mathcal X\times\mathcal Y,\mathcal A)\rightarrow (\mathcal Z,\mathcal H)$.
Regular  sufficiency means that there 
is a function $\mathcal Z \times \mathcal A \rightarrow [0,1]$,
  denoted $(z,A)\mapsto P_z(A)$,
for which 
\begin{enumerate}
\item $P_z$ is a probability measure on $(\mathcal X\times \mathcal Y ,\mathcal A)$ for all $z\in \mathcal Z$;
\item $P_Z(A) $ is a version of $P_\theta(A|Z)$ for all $\theta\in \Theta$
and $A\in \mathcal A$;
\item $P_z(  Z^{-1}\{z\} )=1$ for $\mathcal P$-almost all $z$. 
\end{enumerate}

Recall that $Z$ is boundedly complete if 
for all bounded measurable real-valued functions $f$,
$\Exp{\theta}{ f(Z)} =c $ for all $\theta\in \Theta$ implies 
$f(z)=c$ for
$\mathcal P$-almost all $z$. 
As noted by \citet{faulkenberry_1973}, completeness of $Z$ can be used to relate coverage  
to conditional coverage. In particular, suppose 
$A$ is a $1-\alpha$ constant coverage prediction region, so that 
$P_\theta(A) = 1-\alpha$ for all $\theta$. 
Since $P_\theta(A) = \Exp{\theta}{ P_Z(A) }$, 
completeness of $Z$ implies $P_z(A) = 1-\alpha$ 
for almost all $z$. Combining this with Lemma \ref{lem:svfr}
gives the following characterization of all $1-\alpha$ constant 
coverage prediction regions:
\begin{theorem}
\label{thm:ccrep}
Let $Z$ be a boundedly complete sufficient statistic
for the model $\mathcal P= \{ P_\theta:\theta\in \Theta\}$ on
$(\mathcal X\times \mathcal Y, \mathcal A)$.
A prediction region $A\in \mathcal A$ has constant coverage $P_\theta(A)  = 1-\alpha$ for all $\theta\in \Theta$  if and only if
$A = \{ (x,y): (x,y) \in A_{Z(x,y)} \}$ for a set-valued function $A:\mathcal Z \rightarrow 2^{\mathcal X\times \mathcal Y}$ for which
\begin{enumerate} 
\item $\cup_{z\in \mathcal Z} A_z $ is measurable; 
\item $A_z$ is a subset of $Z^{-1}\{z\}$;
\item $P_z(A_z)=1-\alpha$
for $\mathcal P$-almost all $z$.
\end{enumerate}
\end{theorem}   

Note that 
each $A_z$ can be thought of as the 
acceptance region of a non-randomized 
size-$\alpha$ test of $H_z: (X,Y) \sim P_z$. 
If such a test exists for each $z$, then \emph{all} 
$1-\alpha$ constant coverage prediction regions may be expressed 
as inversions of such tests. 
If the distributions $\{P_z : z\in \mathcal Z\}$ are non-atomic,
then these size-$\alpha$ tests will exist for any choice of $\alpha$. 
Even for some cases where the $P_z$'s are discrete, 
there will exist constant coverage prediction regions for 
certain choices of $\alpha$, as illustrated in the following example:
\begin{example}  
Suppose $Y_1,\ldots, Y_{n+1}$ is an 
independent and identically distributed (i.i.d.)\
random sample 
and we wish to predict $Y_{n+1}$ from a realization of $X=(Y_1,\ldots, Y_n)$.  
Let the model $\mathcal P$ on $\mathcal Y^{n+1}$ be
such that the probability of ties among the $Y_i$'s is zero and 
that the set of unordered values of $Y_1,\ldots, Y_{n+1}$ constitute a 
boundedly complete sufficient statistic. 
This includes several nonparametric families \citep{bell_blackwell_breiman_1960}
as well as many parametric families \citep{mattner_1996}. 
Let $z =\{ y_1,\ldots,y_{n+1}\} \subset \mathcal Y$ be a possible value of 
the sufficient statistic. Then the conditional distribution of  
$(X,Y_{n+1})$ given $Z=z$ has mass $1/(n+1)!$ on each permutation of the vector $(y_1,\ldots, y_{n+1})$. It is therefore possible to find a set $A_z$ for 
$\mathcal P$-almost all $z$ 
such that $P_z(A_z) = 1-\alpha$ for any 
$\alpha=k/(n+1)!$ 
with integer $k$ between zero and $(n+1)!$. 
However, symmetry suggests that 
if $(y_1,\ldots,y_n, y_{n+1}) \in A_z$ then
  $(y_{p_1},\ldots,y_{p_n}, y_{n+1})$ should be 
in $A_z$ as well, for any permutation $p_1,\ldots, p_n$ of $1,\ldots,n$.   
Prediction regions with this kind of symmetry 
can only have constant $1-\alpha$ coverage 
for values $\alpha= k/(n+1)$ with 
$k\in \{0,1,\ldots, n+1\}$. 
\end{example}

In some cases for which the $P_z$'s are discrete there will not exist
constant coverage prediction regions. 
However, completeness can still be used to characterize prediction 
regions that have the same coverage. 
Specifically, 
let $A$ and $A'$ be two 
prediction regions for which $P_\theta(A) =P_\theta(A')$ for all $\theta\in \Theta$. 
Then 
\begin{align}
 0& =   P_\theta(A) -P_\theta(A')  \\ 
  & =       \Exp{ \theta}{ P_Z(A) - P_Z(A') }
\end{align}
for all $\theta$, which implies that $P_z(A) = P_z(A')$ for $\mathcal P$-almost all $z$ 
by the bounded completeness of $Z$. 
Combining this observation with 
Lemma \ref{lem:svfr} gives the following characterization 
of prediction regions with the same coverage:
\begin{theorem}
\label{thm:htrep}  
Let $Z$ be a boundedly complete sufficient statistic
for the model $\mathcal P= \{ P_\theta:\theta\in \Theta\}$ on
$(\mathcal X\times \mathcal Y, \mathcal A)$.
Sets $A$ and $A'$ in $\mathcal A$ 
satisfy $P_\theta(A) = P_\theta(A')$ for all $\theta$
if and only if  
$A= \{ (x,y) : (x,y) \in A_{Z(x,y)} \}$ 
and 
$A'= \{ (x,y) : (x,y) \in A'_{Z(x,y)} \}$
for set-valued functions $A$ and $A'$ mapping
$\mathcal Z \rightarrow 2^{\mathcal X\times \mathcal Y}$ for which 
\begin{enumerate} 
\item $\cup_{z\in \mathcal Z} A_z $ and
       $\cup_{z\in \mathcal Z} A_z' $  are measurable;
\item  $A_z$ and $A_z'$ are  subsets of $Z^{-1}\{z\} $; 
\item $P_z(A_z) =  P_z(A_z')$ for $\mathcal P$-almost all $z$. 
\end{enumerate} 
\end{theorem} 
This result allows us to characterize all prediction regions 
that have coverage equal to that of a given prediction region. 
In the next section, this result will be used to construct regions 
that are Bayes-optimal for their coverage.

\section{Bayes-optimal prediction regions} 

\subsection{Existence and uniqueness of optimal regions} 

As described in the Introduction, 
we define the risk of a prediction procedure $A\in \mathcal A$ 
under $P_\theta$ as 
its expected $\mathcal Y$-volume, so 
that $R_\theta(A) = \Exp{\theta}{ \mu( A_X  ) }$
where $\mu$ is a $\sigma$-finite measure on $(\mathcal Y,\mathcal G)$.    
Letting $P_\theta^X$ be the marginal distribution of $X$ under $P_\theta$, 
for any $A\in \mathcal A$ we have
\begin{align}
R_\theta(A) & = \int  \mu(A_x) \, P^X_\theta(dx)   \\ 
  &  = \int \int  1( y\in A_x) \,\mu(dy) \, P^X_\theta(dx)   \\
  & = \int \int 1( (x,y) \in A) \, P^X_\theta(dx ) \, \mu(dy) 
\end{align}
by Tonelli's theorem, and so 
$R_\theta$ is simply the product measure $P^X_\theta \times \mu$  
on $(\mathcal X\times \mathcal Y, \mathcal A)$. 
Now let $\pi$ be a prior probability measure on the measurable
space $(\Theta,\mathcal T)$ for which $F \mapsto  P_\theta^X(F)$
is a $\mathcal T$-measurable function of $\theta$ for each $F\in \mathcal F$.  
Then $P^X_\pi$,
defined by $P^X_\pi(F) = \int P_\theta^X(F) \, \pi(d\theta)$, is
a probability measure on $(\mathcal X,\mathcal F)$. 
The Bayes risk $R(A)$ of $A\in \mathcal A$ under $\pi$ is then
\begin{align}
R(A) &= \int R_\theta(A) \, \pi(d\theta) \\  
 &= \int \int 1( (x,y) \in A ) \, P_\pi^X(dx) \, \mu(dy),
\end{align}
and so $R$ is the product measure $P_\pi^X \times \mu$ on $(\mathcal X\times \mathcal Y, \mathcal A)$.  
Note that $P_\pi^X$ is finite and $\mu$ is $\sigma$-finite, so 
$R$ is $\sigma$-finite. 

Recall that by the regular sufficiency of $Z$, the coverage of a prediction 
region $A$ can be written as 
\begin{align}
 P_\theta(A) &=
    \int  P_z(A_z)  \,  \nu_\theta(dz),
  \label{eqn:dcover}
\end{align}
where $\nu_\theta$ is the marginal probability measure of $Z$
under $P_\theta$, defined as
$\nu_\theta(H) = P_\theta(Z^{-1}H)$ for $H\in \mathcal H$.
In other words, 
the coverage probability of a region $A$ can be 
represented as an average of conditional probabilities of the sets $\{ A_z = A\cap Z^{-1}\{z\} : \ z\in \mathcal Z\}$. 
We now show how the Bayes risk can 
be similarly represented. 
Let $\nu_R$ be the image measure of $R$ under $Z$, so that 
$\nu_R(H) = R(Z^{-1}H)$ for $H\in \mathcal H$. 
A collection of $\sigma$-finite measures
 $\{  R_z : z\in \mathcal Z  \}$ on 
$(\mathcal X\times \mathcal Y, \mathcal A)$ is a
 \emph{$(Z , \nu_R)$-disintegration} of $R$ if
\begin{enumerate}
\item $z \mapsto R_z(A)$ is measurable for each $A\in \mathcal A$;
\item $R(A) = \int R_z(A) \, \nu_R(dz)$ for each $A\in \mathcal A$.
\item $R_z( Z^{-1}\{z\}) = R_z(\mathcal X\times \mathcal Y)$ for $\nu_R$-almost all $z$;
\end{enumerate}
Additionally, 
$R_z$ is 
a probability measure for $\nu_R$-almost all $z$ if $\nu_R$ is $\sigma$-finite
\citep[Theorem 2]{chang_pollard_1997}. 
A disintegration is a generalization of a 
regular conditional probability distribution to $\sigma$-finite measures that are not necessarily probability measures, 
such as $R$ for some choices of $\mu$. 
If a $(Z, \nu_R)$-disintegration of $R$ exists, we can write 
\begin{align}
R(A)  = \int R_z(A) \, \nu_R(dz) 
      = \int  R_z(A_z) \, \nu_R(dz) 
  \label{eqn:drisk}
\end{align}
where the last equality holds 
because $R_z$ has mass only on $Z^{-1}\{z\}$ 
for almost all $z$. 

Equations \ref{eqn:dcover} and \ref{eqn:drisk} provide a means 
of constructing a $1-\alpha$ constant coverage prediction region 
that is optimal in terms of the Bayes risk $R$, 
assuming the set of regions with this coverage is not empty. 
We first describe the idea informally:
By completeness of $Z$, 
a region $A$  for which 
 $P_\theta(A) = 1-\alpha$ for all $\theta$ 
must also satisfy 
$P_z(A_z) = 1-\alpha$ for $\mathcal P$-almost all $z$.  
To construct a set $A^\pi$ that is optimal among such regions, 
simply let $A_z^\pi$ be the subset of $Z^{-1}\{z\}$ that has minimal 
$R_z$-measure among sets with $P_z$-measure equal to $1-\alpha$. 
By (\ref{eqn:drisk}), the region $A^\pi=\cup_z A^\pi_z$ will 
then have minimal $R$-measure among 
regions with $1-\alpha$ constant coverage. 

This line of reasoning can also be used to construct regions 
that lack constant coverage 
(as will typically be the case for models having discrete distributions), 
but are nevertheless Bayes-optimal in the sense that they 
minimize the Bayes risk among regions having the same coverage 
as each other, 
as a function of $\theta$. Generally speaking, a region $A^\pi$ 
will be Bayes-optimal for its own coverage function 
if for each $z$, $R_z(A_z)\geq R_z(A_z^\pi)$ 
for any $A_z$ such that $P_z(A_z) = P_z(A_z^\pi)$.  
In other words, a
subset $A_z^\pi$ of $Z^{-1}\{z\}$ is optimal if
no sets having greater or equal $P_z$-coverage have smaller $R_z$-risk. 
The form of 
such a minimizing set is given by the following variant of the 
Neyman-Pearson lemma: 
\begin{lemma}
\label{lem:nplemma}
Let $P$ and $R$ be probability measures on $(\Omega, \mathcal A)$, 
and let $\tilde A$ have the form
\begin{equation}
 \tilde A = \{ \omega: p(\omega) > k r(\omega)  \}
\end{equation}
where $k>0$ and $p$ and $r$ are densities of $P$ and $R$
with respect to a common dominating measure. 
Then 
 $R(A ) \geq R(\tilde A)$ 
for all $A\in \mathcal A$ such that $P( A) \geq P(\tilde A)$, 
with equality  only if  
 $P( A \Delta\tilde A)=0$.  
\end{lemma}

Applying this lemma to $P_z$ and $R_z$ for each $z$ 
gives an $R_z$-optimal set $A_z^\pi$ for each $z$. Combining these 
gives an $R$-optimal set $A^\pi= \cup A_{z}^\pi$:
\begin{theorem} 
\label{thm:absthm}
Let $\mathcal P= \{P_\theta:\theta\in \Theta\}$ be a model
  on $\{\mathcal X\times \mathcal Y, \mathcal A\}$
with a boundedly complete regular sufficient statistic
 $Z:( \mathcal X\times \mathcal Y , \mathcal A)\rightarrow (\mathcal Z, \mathcal H)$ having marginal distributions $\{ \nu_\theta:\theta\in \Theta\}$
on $(\mathcal Z, \mathcal H)$. 
Let $R$ be a $\sigma$-finite measure on $(\mathcal X\times \mathcal Y, \mathcal A)$ with a $\sigma$-finite image measure $\nu_R$ on 
$(\mathcal Z, \mathcal H)$ and a $(Z,\nu_R)$-disintegration 
$\{ R_z:z\in \mathcal Z\}$. 
Assume that  
\begin{enumerate}
\item $\nu_\theta(H) = 0$ for all $\theta$ implies  $\nu_R(H)=0$;
\item $\nu_\theta \ll \nu_R$ for all $\theta$.
\end{enumerate}
 Let $A^\pi\in \mathcal A$ 
have the form 
\begin{equation}
A^\pi \cap Z^{-1}\{z\}\equiv A^\pi_z = \{ (x,y)\in Z^{-1}\{z\} : p_z(x,y)  > k_z r_z(x,y)   \} 
\label{eqn:bopt}
\end{equation} 
where $k_z>0$ and $p_z$ and $r_z$ are densities of $P_z$ and $R_z$ 
with respect to a common dominating measure.  
Let $A\in \mathcal A$ be such that  $P_\theta(A) = P_\theta(A^\pi)$ for all $\theta\in \Theta$. Then 
$R(A) \geq  R(A^\pi)$,
with equality only if $R(A^\pi) = \infty$ or 
 $P_\theta(A\Delta A^\pi)=0$ for all $\theta$. 
\end{theorem} 
The prediction region $A^\pi$ is a Bayes procedure in the sense that
it minimizes a Bayes risk, but it is also frequentist in the sense that
its optimality is among procedures having the same 
frequentist coverage rate. Following \citet{yu_hoff_2018}, 
we refer to such a procedure as being ``frequentist and Bayesian'', or FAB. 

Before obtaining a less abstract  form for the optimal sets $\{ A_z^\pi:z\in \mathcal Z\}$ and the resulting FAB prediction region $A^\pi$, 
we first discuss some of the conditions of the theorem. 
The starting point 
is the existence of a disintegration 
of $R$. 
As discussed 
in \citet{chang_pollard_1997}, the existence of a disintegration 
typically 
requires some topological assumptions about the underlying spaces.
While not as general as it could be, 
the following corollary of their Theorem 1 is
sufficient for many applications. In particular, 
the conditions on $\mathcal X$, $\mathcal Y$ and $\mathcal Z$ are
met by Euclidean spaces, most countable spaces, many topological manifolds, and
products of these spaces.

\begin{corollary}[Corollary of {\citet[Theorem 1]{chang_pollard_1997}}]
\label{cor:cpcor} 
Let $\mathcal X$ and $\mathcal Y$ be complete, separable and locally compact 
metric spaces, with $\mathcal F$ and $\mathcal G$ being the 
Borel sets.  
Let $P_\pi^X$ be a probability measure on 
$(\mathcal X, \mathcal F)$, and let $\mu$ be a $\sigma$-finite Radon measure 
on $(\mathcal Y, \mathcal G)$.  
Let $\mathcal A=\mathcal F\otimes \mathcal G$ 
and define the product measure 
$R=P_\pi^X \times \mu$ on $(\mathcal X\times\mathcal Y, \mathcal A)$.  
Then $R$ is a $\sigma$-finite Radon measure. Additionally,  
let $Z:(\mathcal X\times \mathcal Y, \mathcal A)
\rightarrow (\mathcal Z, \mathcal H)$ where $\mathcal Z$ is a
separable metric space and $\mathcal H$ is the Borel $\sigma$-algebra.
If the image measure $\nu_R$ of $R$ under $Z$ 
is $\sigma$-finite, then $R$ has a $(Z,\nu_R)$-disintegration.
\end{corollary}

We caution that
the image measure $\nu_R$ can fail to be $\sigma$-finite if 
$P_\pi^X$ is not a proper probability measure, even if 
$R$ is $\sigma$-finite. 
For example, consider the case that $\mathcal X =\mathcal Y = \mathbb R$, 
$P_\pi^X$ and $\mu$ are both Lebesgue measure, and $Z=X+Y$. 
Then $\nu_R( [a,b] ) =\infty$ for all $a<b$. However, if $P_\pi^X(\mathbb R)=1$ (or is finite) then $\nu_R$ is $\sigma$-finite.

Conditions 1 and 2 of Theorem \ref{thm:absthm}
concern the relative 
absolute continuity of $\nu_R$ and $\{ \nu_\theta: \theta\in\Theta\}$. 
Condition 1 
roughly means that $z$-values that 
are impossible under the model should not contribute to the risk 
of a prediction procedure. 
More specifically, let $N\in \mathcal H$ be such that 
$\nu_\theta(N) = 0$ for all $\theta$. 
A competitor $A$ to $A^\pi$ does not need to maintain 
$P_z(A_z) = P_z(A_z^\pi)$ for $z\in N$ in order to maintain 
$P_\theta(A) = P_\theta(A^\pi)$ for all $\theta$, and so  it could be that
$R_z(A_z) < R_z(A^\pi_z)$ for $z\in N$. Without the condition, 
if $\nu_R(N)>0$
it is possible that $R(A) < R(A^\pi)$. Conversely, with the condition we have 
$\nu_R(N)=0$, and so $R(A) \geq  R(A^\pi)$. 
Condition 2 of the theorem is not completely necessary,
 but it does imply 
that  $A^\pi$ is essentially unique, 
in that any other set with the same coverage and 
risk as $A^\pi$ can
only differ from $A^\pi$ by a set of measure zero. 

\subsection{Optimal regions for a  given level}
Recall that by Theorem \ref{thm:ccrep}, every prediction region with
constant coverage also has constant conditional coverage.
Theorem \ref{thm:absthm} then implies that
if $k_z$ in (\ref{eqn:bopt}) is chosen to yield a conditional coverage rate  of $1-\alpha$
for all $z\in \mathcal Z$,
then a Bayes-optimal $1-\alpha$ constant coverage
prediction region may be obtained. Specifically, 
if 
 $A^{\pi}_z$ is defined as in (\ref{eqn:bopt}), with
$k_z$  chosen to satisfy
  $P_z(A^\pi_z) = 1-\alpha$ for 
all $z$, then the FAB region $A^\pi$ has minimum Bayes risk
among all
 $1-\alpha$ constant
coverage prediction regions.

This result does not by itself imply that 
 $A^\pi$ is optimal among regions with 
non-constant coverage of 
 $1-\alpha$ or greater, that is, regions 
$A$ such that $P_\theta(A)\geq 1-\alpha$ for all $\theta$
with inequality for some $\theta$. 
More generally, 
consider the risk optimality of 
$A^\pi$ among prediction regions $A$ for 
which $P_\theta(A) \geq P_\theta(A^\pi)$ for all $\theta\in \Theta$, 
with inequality for some $\theta$. Intuitively we expect 
$R(A)$ to be larger than $R(A^\pi)$, since the larger 
coverage probability of $A$ should correspond to a larger volume, and hence 
a larger risk.  However, as discussed in Section 2.1, 
completeness
of $Z$ does not rule out the possibility that 
$P_z(A) < P_z(A^\pi)$ for some $z$ 
even if $P_\theta(A) \geq  P_\theta(A^\pi)$ for all $\theta$, 
and so it is possible that $R_z(A)< R_z(A^\pi)$ for 
these values of $z$.
As a result,  $R(A)$ could be smaller than $R(A^\pi)$,
depending on where the image measure $\nu_R$ on $(\mathcal Z,\mathcal H)$ places mass. 

However, 
some models $\mathcal P$ are rich enough so that 
$P_\theta(A) \geq P_\theta(A^\pi)$ for all $\theta\in \Theta$
\emph{does} imply that $P_z(A) \geq P_z(A^\pi)$ for 
 $\mathcal P$-almost all $z$, in which case 
$A^\pi$ is risk-optimal among all procedures 
with equal or greater coverage. 
Models for which this is the case are those for which 
the set of image probability measures 
$\{ \nu_\theta : \theta\in \Theta\}$ on $(\mathcal Z, \mathcal H)$ 
has elements with arbitrarily high concentration on subsets 
of $\mathcal Z$: 
\begin{lemma}
\label{lem:excomplete}
Suppose for every $\epsilon>0$ and
$\{ \nu_\theta:\theta\in \Theta\}$-non-null set $H\in \mathcal H$ there exists
a $\theta_{\epsilon}\in \Theta$ such that $\nu_{\theta_\epsilon}(H) > 1-\epsilon$.
Then if $h:(\mathcal Z,\mathcal H) \rightarrow (\mathbb R, \mathcal B(\mathbb R))$ 
is a bounded function for which
$\int h(z) \, \nu_\theta(dz) \geq 0$ for all $\theta\in \Theta$, 
then
$h(z) \geq 0$ for $\nu_\theta$-almost all $z$, for every $\theta\in \Theta$. 
\end{lemma}

From this lemma we have the following result on the risk optimality 
of $A^\pi$:
\begin{theorem} 
\label{thm:boptlevel}
Under the conditions of 
Theorem \ref{thm:absthm} and Lemma \ref{lem:excomplete}, 
if $A\in \mathcal A$ satisfies 
$P_\theta(A)\geq P_\theta(A^\pi)$ for all $\theta\in \Theta$, then 
$R(A) \geq  R(A^\pi)$, 
that is, $A^\pi$ is risk-optimal among prediction 
regions with equal or greater coverage. 
\end{theorem}
Models for which the conditions of the theorem hold include those for 
which the location and scale of the distribution of $Z$ can be 
set arbitrarily, such as 
multiparameter exponential families.

\subsection{Expressions for optimal regions}
We now obtain some less abstract expressions for 
the optimal sets $\{ A_z^\pi : z\in \mathcal Z\}$ given by (\ref{eqn:bopt}), in 
the case that the risk measure $R$ dominates 
the model $\{ P_\theta : \theta\in \Theta\}$.  
In this case, it will be shown that each $R_z$ dominates 
the corresponding conditional 
distribution $P_z$, and so the optimal sets given by 
(\ref{eqn:bopt}) may equivalently be expressed as 
subsets of $Z^{-1}\{z\}$ 
such that $dP_z/dR_z$ exceeds some threshold. 
Each $A_z^\pi$ may therefore be recognized as the acceptance region of a test 
of $H_z:(X,Y)\sim P_z$ versus $K_z: (X,Y)\sim R_z$
using the most powerful test statistic $dP_z/dR_z$.

Without loss of generality, 
assume $P_\pi^X$ is dominated by a $\sigma$-finite measure $\lambda$
on $(\mathcal X, \mathcal F)$, with density 
 $p_\pi^X(x) >0$ for all $x\in \mathcal X$ 
(we could take $\lambda$ to be $P_\pi^X$). 
Then 
$R \equiv P_\pi \times \mu  \ll \lambda \times \mu$ with density $p_\pi^X$.  
Since $P_\theta\ll R \ll \lambda\times \mu$ for each $\theta$, 
 by the factorization theorem
 \citep{halmos_savage_1949}, there
exist functions $g$ and $\{h_\theta:\theta\in \Theta\}$ such that
\begin{equation} p_\theta(x,y)\equiv [dP_\theta/d(\lambda\times \mu)](x,y) =  h_\theta(Z(x,y)) \times g(x,y)  .\end{equation}
Therefore, 
\begin{equation}
  dP_\theta/dR = h_\theta(Z(x,y)) \times g(x,y)/p_\pi^X(x). 
\end{equation}
The following theorem shows that the density of each $dP_z/dR_z$ 
has a similar form:
\begin{theorem} 
\label{thm:rnderiv}
Let $\lambda$ and $\mu$ be $\sigma$-finite measures 
on $(\mathcal X,\mathcal F)$ and $(\mathcal Y, \mathcal G)$ respectively, 
with product measure $\lambda\times \mu$ on $(\mathcal X\times \mathcal Y, \mathcal F\otimes \mathcal G)$. 
Let 
$R$ be a measure and $\{ P_\theta:\theta\in \Theta\}$ be a model 
on $(\mathcal X\times \mathcal Y, 
\mathcal F\otimes \mathcal G)$, the latter having 
regular sufficient statistic $Z:(\mathcal X\times \mathcal Y, \mathcal F\otimes \mathcal G)\rightarrow (\mathcal Z,\mathcal H)$. 
If
\begin{enumerate}
\item $\{ P_\theta:\theta\in \Theta\} \ll R \ll  
  \lambda\times \mu$; 
\item $dR/d(\lambda\times \mu) \equiv p_\pi^X >0$ with $\int p_\pi^X(x) \, \lambda(dx) =1$; 
\item the image measure $\nu_R$ of $R$ under $Z$ is $\sigma$-finite;
\item $R$ has $(Z,\nu_R)$-disintegration  $\{ R_z:z\in \mathcal Z\}$, 
\end{enumerate} 
then there is a common regular conditional distribution 
$\{ P_z:z\in \mathcal Z\}$ for each element of 
$\{P_\theta:\theta\in \Theta\}$ such that 
$P_z \ll R_z$ for all $z$ with density
\begin{equation}
dP_z/dR_z = c(z) \times  g(x,y)/ p_\pi^X(x), 
\label{eqn:dpdr}
 \end{equation}
where $c(z)$ is a positive function and $g(x,y)$ is defined by the 
 factorization 
$dP_\theta/d(\lambda\times \mu) = h_\theta(Z(x,y)) \times g(x,y)$. 
\end{theorem}
These results are related to 
Theorem 3 of \citet{chang_pollard_1997}, which describes 
how the disintegration of one measure can be related to the disintegration 
of another measure that dominates it. 
We note that 
completeness is not used in this lemma, and that 
the 
existence of a common regular conditional distribution for the model 
is simply the definition of $Z$ being a regular sufficient statistic. 
What the lemma provides is the existence of a conditional distribution 
that is dominated by the disintegration of the risk measure, and 
an expression for the corresponding densities $\{ dP_z/dR_z: z\in 
\mathcal Z\}$. 
This expression provides the following 
version of Theorem \ref{thm:absthm}, written in terms of the 
somewhat familiar quantities $g(x,y)$ and $p_\pi^X(x)$:
\begin{corollary}  
\label{cor:nfrep}
Under the conditions of 
Theorem \ref{thm:absthm} and 
Theorem \ref{thm:rnderiv}, 
let $A^\pi \in \mathcal F\otimes \mathcal G$ satisfy
\begin{equation}
A^\pi_z = \{ (x,y)\in Z^{-1}\{z\} : g(x,y)/p_\pi^X(x)  > k_z   \}. 
\label{eqn:rnbopt}
\end{equation} 
Let $A\in \mathcal A$ be such that  $P_\theta(A) = P_\theta(A^\pi)$ for all $\theta\in \Theta$. Then 
$R(A) \geq  R(A^\pi)$,
with equality only if 
 $P_\theta(A\Delta A^\pi)=0$ for all $\theta$. 
\end{corollary}

Practical details concerning the constructing of 
the FAB region $A^\pi$ 
for a few specific models
are described in Section 4, but we make some comments here. 
Having observed $X=x$, 
construction of
a $1-\alpha$ FAB prediction region 
amounts to
determining the values $y\in \mathcal Y$ for which 
$g(x,y)/p_\pi^X(x)$ meets or exceeds the $\alpha$ quantile
of $g(X,Y)/p_\pi^X(X)$ under $P_z$, 
the conditional distribution of $(X,Y)$
given $Z=z$, where $z=Z(x,y)$. 
Thus the region is constructed by inverting tests
of $(X,Y) \sim P_z$ using the test statistic
$t_\pi(x,y) = g(x,y)/p_\pi^X(x)$, and so we refer to 
$t_\pi(x,y)$ as being a Bayes-optimal prediction statistic 
under the prior distribution $\pi$. 
In some problems it will be more convenient or familiar to work with 
an alternative statistic that corresponds to the same Bayes-optimal tests. 
For example, any statistic that is a strictly increasing function 
of $t_\pi(x,y)$ is also a Bayes-optimal prediction statistic, 
as is any statistic that can be expressed as 
$h(z) \times t_\pi(x,y)$ for some positive function $h(z)$. 
This implies, for example, that 
$p_{\theta_0}(x,y)/p_\pi^X(x)$ 
is a Bayes-optimal prediction statistic
for any choice of $\theta_0\in \Theta$.

Construction of the FAB prediction region $A^\pi$ 
may be simplified in the case that 
the function 
$Z_x:\mathcal Y\rightarrow \mathcal Z$ defined by  $Z_x(y)= Z(x,y)$
is injective. 
As described in Corollary \ref{cor:accF}, if
$Z_x$ is injective for each $x$ then any prediction region $A$
can be expressed as
$A = \{ (x,y) : x\in B_{Z(x,y)} \}$ for some $B:\mathcal Z\rightarrow 2^\mathcal X$, and therefore derived by inverting 
tests of $H_z:X\sim P_z^X$. 
In some cases, the optimal tests 
are equivalent to most powerful tests of $H_z:X\sim P_z^X$
 versus $K_z:X\sim P_\pi^X$, and so 
are those that 
accept $H_z$ for large values of $p^X_z(x)/p^X_\pi(x)$,
where $p^X_z$ is the density of the conditional distribution of $X$ given 
$Z=z$. We first illustrate this with an example, and then give some general results.

\begin{example}
\label{ex:gaussianOR}
Consider again the model in Example \ref{ex:simplenormal} where
$X\sim N(\theta,k\sigma^2)$ and $Y\sim N(\theta, \sigma^2)$ are
independent, with $\sigma^2$ known and $\theta\in \mathbb R$ unknown. 
Taking the volume measure $\mu$ to be Lebesgue measure on $\mathbb R$, 
$t_\pi(x,y)$ can be written as
\begin{align}
t_\pi(x,y) & =
  \exp(  -\tfrac{1}{2\sigma^2} [x^2/k + y^2] )/p_\pi^X(x). 
\end{align}
Further simplification is possible in this
case. 
For $z=(x+yk)/(1+k)$, 
$t_\pi(x,y)$ can be written
as
\begin{equation} 
t_\pi(x,y) = h(z) \exp(  -\tfrac{k+1}{2k^2\sigma^2} (x -z)^2 )/p_\pi^X(x), 
\end{equation}
which has a numerator that is proportional to the density of 
the $N( z , \sigma^2 k^2/(k+1) )$ distribution - the  distribution of $X$ given $Z=z$.  
Therefore, having observed $X=x$, the 
Bayes-optimal 
$1-\alpha$ constant coverage prediction region includes $y$ if
$p_z^X(x)/p_\pi^X(x)$ exceeds its $\alpha$ quantile
under $X\sim P_z$, where $z=Z(x,y)$ and $p_z^X$ is the conditional density of $X$
given $Z=z$. In other words, a Bayes-optimal prediction region 
with constant $1-\alpha$ coverage
is
$A^\pi= \{ (x,y): x\in B^\pi_{Z(x,y)} \}$ where  $B^\pi_z$ can be expressed as
\begin{equation}
 B^\pi_z = \{ x: p_z^X(x)/p_\pi^X(x) > k_z \},
\end{equation}
with $k_z$ being the $\alpha$ quantile of $p^X_z(X)/p_\pi^X(X)$
under $X\sim N(z, \sigma^2 k^2/(k+1) )$.
Thus $B^\pi_z$ is the acceptance region of the
most powerful level-$\alpha$ test of $H_z: X \sim P^X_z$ versus
$K: X\sim P^X_\pi$.
A reasonably simple formula for the prediction region
under a conjugate prior distribution for $\pi$ is given in the next section. 
\end{example}

A similar result holds for 
other models where $Z_x$ is injective. 
For example, 
if the distributions $\{P_\theta:\theta\in \Theta\}$ are all 
dominated by counting measure with densities $p_\theta(x,y) = h_\theta(Z(x,y)) g(x,y)$, then the conditional density of $X$ given $Z=z$ is 
\begin{align}
p_z^X(x)&  =  \frac{ \sum_{y':Z(x,y')=z} h_\theta(Z(x,y')) g(x,y') }
             { \sum_{(x',y'):Z(x',y')=z} h_\theta(Z(x',y')) g(x',y')  } \\
 &=  \sum_{y':Z(x,y')=z} g(x,y')/c_z  \\
 & =  g(x,Z_x^{-1}(z))/c_z,
\end{align}
where the last equality holds by the injectivity of $Z_x$. This means 
that, on $\{ (x,y): Z(x,y)=z\} $, we have $g(x,y) =  c_z p_z^X(x)$. 
By plugging this into 
(\ref{eqn:rnbopt}) we may express the FAB prediction region 
as 
$A^\pi =\{ (x,y) : x \in B^\pi_{Z(x,y)} \}$, where 
\begin{equation}
 B^\pi_z = \{ x: p_z^X(x)/p_\pi^X(x) > k_z\}.   
\label{eqn:dibopt}
\end{equation} 
Alternatively, if $\{P_\theta:\theta\in \Theta\}$ 
has densities with respect to Lebesgue measure and $Z(x,y)$ is 
differentiable, then via the usual change of variables 
formula the density of $(X,Z)$ under $P_\theta$ is
\begin{align} 
 p_\theta^{X,Z}(x,z)  &=  p_\theta(x,Z^{-1}_x(z) )/ J(x,z) \\
 &=    h_\theta(z) \times  g(x,Z^{-1}_x(z) )   / J(x,z), 
\end{align}
where $J(x,z)$  is $| dZ(x,y)/dy |$ evaluated at $(x,Z_x^{-1})$, and so
 the conditional density of $X$ given $Z=z$  is
\begin{equation} p_z^X(x) = c_z \times g(x,Z_x^{-1}(z))/ J(x,z). \end{equation}
Therefore, on $Z(x,y)=z$, 
the Bayes-optimal test statistic (\ref{eqn:dpdr})
is equal to $J(x,z) \times p_z^X(x)/p^X_\pi(x)$, 
and the FAB region  may be expressed as 
$A^\pi = \{ (x,y) : x \in B^\pi_{Z(x,y)} \}$, where 
\begin{equation}
 B^\pi_{z} = \{ x : J(x,z) \times p^X_z(x)/p^X_\pi(x) > \tilde k_z\}.  
\label{eqn:cibopt} 
\end{equation}
If $J(x,z)$ is constant in $x$, then $B_z^\pi$ can be written  
$B^\pi_{z} = \{ x : p^X_z(x)/p^X_\pi(x) > k_z\}$, just as with discrete 
models. 
This will be the case if $Z$ has the form
$Z(x,y) = s^X(x) +s^Y(y)$, as with 
many exponential family models for independent $X$ and $Y$. 

To summarize, 
when $Z_x$ is injective the optimal test statistic (\ref{eqn:dpdr})
may be written as a function of $x$ and $z$, and so the optimal 
acceptance regions may be written as acceptance regions of 
tests of $H_z:X\sim P_z^X$. 
The inclusion of a value $y$ into a prediction region 
based on observing $X=x$ can be
determined as follows:
\begin{enumerate}
\item Set $z=Z(x,y)$. 
\item Test $H_z:X\sim P_z^X$ based on observing $X=x$.
\item Include $y$ in the region if $H_z$ is accepted, otherwise 
exclude $y$ from the region. 
\end{enumerate}
The coverage of such a region as a function of $\theta$ is 
$1-\Exp{\theta }{a(Z)}$, where $a(z)$ is the size of the test 
in Step 2. The region will be Bayes-optimal for its coverage function 
if the statistic used in Step 2 
is one that is  equivalent to (\ref{eqn:dpdr}), which 
in some cases includes
$p^X_z(x)/p^X_\pi(x)$ or $J(x,z)\times p^X_z(x)/p^X_\pi(x)$.

\section{Examples}
\subsection{Nonparametric prediction} 
Conformal prediction \citep{gammerman_vovk_vapnik_1998,shafer_vovk_2008} 
is a method of constructing nonparametric prediction 
regions 
for a random object $Y_{n+1}$ 
based on $X=( Y_1,\ldots, Y_n)$, 
in models for which $Y_1,\ldots, Y_{n+1}$ are exchangeable. 
A generic conformal prediction region is constructed as follows:
Let $c:\mathcal Y^{n+1}\rightarrow \mathbb R$ be
a function 
that is invariant to permutations of its first $n$ elements. 
Typically, 
$c(y_1,\ldots,y_{n+1})$ is chosen to be some 
numerical measure of ``conformity'' between $y_{n+1}$ 
and the (multi)set $\{y_1,\ldots, y_n\}$.
Having observed $(Y_1,\ldots, Y_n ) = (y_1,\ldots, y_n)$, 
a value $y_{n+1}$ is included in the
prediction region if the conformity between $y_{n+1}$ and 
$\{ y_1,\ldots, y_n\}$ is comparable to 
the conformity between $y_i$ and $\{ y_1,\ldots,y_{i-1},y_{n+1},y_{i+1},
\ldots, y_n\}$ for some minimal fraction of indices $i\in \{1,\ldots, n\}$.  
Specifically, a value $y_{n+1}$ is included 
in the 
$1-\alpha$ 
conformal prediction region if $c_{n+1}$ is greater than 
the $\alpha$ sample quantile of 
$\{c_1,\ldots,  c_{n+1}\}$, where 
$ c_i=c( y_1,\ldots, y_{i-1},y_{n+1},y_{i+1},\ldots, y_n, y_i)$. 
By exchangeability, the coverage probability of the resulting prediction 
region is greater than or equal to $1-\alpha$. 

Conformal prediction is 
a special case of Faulkenberry's method, 
as applied to 
exchangeable models. 
For notational simplicity, we consider the i.i.d.\ case where 
 $Y_1,\ldots, Y_{n+1}  \sim$ i.i.d.\ $P\in \mathcal P$
 where $\mathcal P$ is a model on $(\mathcal Y,\mathcal G)$, 
so that the $(n+1)$-fold product measure 
$P^{n+1}\equiv\times_{1}^{n+1} P$ is 
the joint distribution of $(Y_1,\ldots, Y_{n+1})$. 
Then 
the multiset  $\{Y_1,\ldots, Y_{n+1}\}$ 
of unordered observed values  (the ``order statistics'')
is a sufficient statistic. 
Faulkenberry's prediction method in this case is 
that, having observed $(Y_1,\ldots, Y_n)=(y_1,\ldots, y_n)$, 
a value $y_{n+1}$ is included in the prediction region 
if 
$y_{n+1}$
is in the acceptance region of a level-$\alpha$ 
test of $Y_{n+1}$ being uniformly distributed on 
$\{ y_1,\ldots,y_{n+1}\}$, 
as this is the conditional 
distribution of $Y_{n+1}$ given 
$\{ Y_1,\ldots,Y_{n+1} \} = \{y_1,\ldots, y_{n+1}\}$, for 
any member $P$ of $\mathcal P$. 
A  test that accepts values $y_{n+1}$ for which 
$c_{n+1}$ is greater than the $\alpha$ sample quantile of
$\{c_1,\ldots,  c_{n+1}\}$, 
is equivalent to the conformal prediction procedure.

For many models the unordered values $\{Y_1,\ldots, Y_{n+1}\}$
are not only a sufficient statistic, but a complete sufficient statistic
\citep{bell_blackwell_breiman_1960,mattner_1996}. 
In particular, this will be the case 
if $\mathcal P$ is the space of probability measures 
dominated by a common measure $\mu$, where 
 $\mu$ is non-atomic (e.g.\ Lebesgue measure) or $\mathcal Y$ is  countable.
In
these cases, the class of all prediction regions with a
given coverage function may be characterized in terms of
conditional coverage given the order statistics, and under some regularity conditions a Bayes-optimal FAB prediction region
may be found. 
We first derive an expression for the region 
assuming that the conditions of Corollary \ref{cor:nfrep} 
are met, and then we discuss 
the conditions.  

Recall that Corollary \ref{cor:nfrep} 
says that a value $y$ is accepted into the FAB prediction region 
if $g(x,y)/p_\pi^X(x)$ is large compared to its conditional 
distribution given the sufficient statistic. 
In the i.i.d.\ sampling case being considered here 
where $X=(Y_1,\ldots, Y_n)$ and 
$Y=Y_{n+1}$,
the function $g(x,y)$ from the factorization theorem is 
constant, and so a risk-optimal prediction region is one that
accepts a value $y_{n+1}$
if $p_\pi(y_1,\ldots, y_n)$ is small compared to the conditional 
distribution 
of $p_\pi(Y_1,\ldots, Y_n)$ given $\{Y_1,\ldots, Y_{n+1}\} = 
\{ y_1,\ldots, y_{n+1}\}$, 
where $p_\pi(y_1,\ldots, y_n)$ is 
the prior predictive density
of $(Y_1,\ldots, Y_n)$ under the prior distribution $\pi$ on $\mathcal P$.  
Recall from the discussion in Section 3.3 that an 
equivalent criterion is to accept $y_{n+1}$ if 
$h(z)/p_{\pi}(y_1,\ldots,y_n)$ is large, where $h(z)$ is any function 
of the the sufficient statistic $z=\{ y_1,\ldots, y_{n+1}\}$. One such 
function is 
the prior predictive density of $Y_1,\ldots, Y_{n+1}$, denoted by 
$p_\pi(y_1,\ldots, y_{n+1})$, which is constant on
$\{Y_1,\ldots,Y_{n+1} \} = \{ y_1,\ldots, y_{n+1}\}$ 
because prior predictive distributions under i.i.d.\ sampling 
are exchangeable. Therefore, 
the FAB prediction region is one that accepts 
values $y_{n+1}$ for which the posterior predictive density $p_\pi(y_{n+1} | y_1,\ldots, y_n) = p_\pi(y_1,\ldots, y_{n+1} ) / p_\pi(y_1,\ldots,y_n)$ is large. 

More concretely, construction of a risk-optimal $1-\alpha$ prediction region 
proceeds as follows:
The conditional distribution of 
$p_\pi(Y_{n+1} | Y_1,\ldots, Y_n)$ given 
$\{ Y_1,\ldots, Y_{n+1}\}
= \{y_1,\ldots, y_{n+1}\}$
has mass 
$1/(n+1)$ on each of 
the values  $c_1,\ldots, c_{n+1}$, where
$c_i = p_\pi( y_{i} | y_1,\ldots, y_{i-1},y_{n+1},y_{i+1},\ldots, y_n)$
for $i=1,\ldots,n$  and $c_{n+1} = p_\pi(y_{n+1} | y_1,\ldots, y_n)$. 
If $\alpha = k/(n+1)$
for some integer $k$ between zero and $n+1$
and there are no ties, 
then $y_{n+1}$ is accepted into the 
$1-\alpha$ FAB prediction region if 
$c_{n+1}$ 
is greater than the $k$th order statistic of $c_1,\ldots, c_{n+1}$. 
This is 
equivalent to 
 implementing the conformal prediction procedure using 
the posterior predictive density $p_\pi(y_{n+1}| y_1,\ldots, y_n)$ 
as the conformity function. 

We now provide some conditions under which 
the procedure described 
above provides a risk-optimal prediction region. 
In doing so, we make use of the fact that the sufficient statistic may be equivalently be expressed as the empirical distribution $Z$ 
of $Y_1,\ldots, Y_{n+1}$, so $[Z(y_1,\ldots, y_{n+1})](G) \equiv \sum_{i=1}^{n+1} 1(y_i\in G)/(n+1)$ for $G\in \mathcal G$. 
\begin{theorem}
\label{thm:nppred} 
Let $\mu$ be a $\sigma$-finite Radon measure on 
$(\mathcal Y, \mathcal G)$, where $\mathcal Y$ 
is a complete, separable and locally compact metric space
and $\mathcal G$ is the Borel $\sigma$-algebra. 
Let  $(\mathcal P,\mathcal \sigma(w) )$ be the 
measurable space of 
probability measures 
    on $(\mathcal Y, \mathcal G)$ that are dominated by $\mu$
  with $\sigma(w)$ being the Borel $\sigma$-algebra under the weak topology. 
Let $\pi$ be a probability measure over $(\mathcal P,\sigma(w))$ 
and let $P_\pi^X$ be the probability measure 
on the $n$-fold product space 
$(\mathcal Y^n, \mathcal G^n)$ of $(\mathcal Y, \mathcal G)$, 
defined by $P_\pi^X(G_1\times \cdots \times G_n) = 
 \int \prod_{i=1}^n P(G_i) \, \pi(dP)$. 
Then 
\begin{enumerate}
\item $R=P_\pi^X\times \mu$ is a $\sigma$-finite Radon measure; 
\item the image measure $\nu_R$ of $R$ under $Z$ is $\sigma$-finite; 
\item $R$ has a $(Z,\nu_R)$-disintegration. 
\end{enumerate} 
If additionally 
\begin{enumerate}
\item[A1.] $\mu$ is non-atomic, or $\mathcal Y$ is discrete and $\mu$ is  
counting measure, and 
\item[A2.] $P^X_\pi$ and  $\times_{i=1}^n \mu$ 
are mutually absolutely continuous, 
\end{enumerate}
then 
$Z$ is a complete regular  sufficient statistic and 
a prediction region given by 
\begin{equation}
A^\pi_{(y_1,\ldots, y_n)}  = 
\{  y_{n+1} : p_\pi(y_{n+1}| y_1,\ldots, y_n ) > k_{Z(y_1,\ldots, y_{n+1})} \}.
\end{equation}
is
risk-optimal 
in that if 
$P^{n+1}(A) = P^{n+1}(A^\pi)$ for all $P\in \mathcal P$ 
then $R(A) \geq R(A^\pi)$, with equality only if 
$P^{n+1}( A\Delta A^\pi) = 0 $ for all $P\in \mathcal P$. 
\end{theorem} 
Some of these conditions may be relaxed. As discussed in 
\citet{chang_pollard_1997}, disintegrations of $R$ may exist under other conditions on $(\mathcal Y, \mathcal G)$. Also, $Z$ is  a complete sufficient statistic for models other than the ones mentioned in assumption A1, as described in 
\citet{bell_blackwell_breiman_1960} and \citet{mattner_1996}.

We comment that Theorem \ref{thm:boptlevel},
regarding optimality  of $A^\pi$ among procedures with 
coverage that is bounded below by $1-\alpha$, 
 does not apply to these nonparametric models
because 
under i.i.d.\ sampling the 
distributions $\{ \nu_P : P\in \mathcal P\}$  of $Z$ cannot put mass on 
arbitrarily small subsets of $\mathcal Z$.  However, 
a constant coverage 
conformal prediction procedure using $p(y_{n+1}| y_1,\ldots, y_n)$
as the conformity score will be optimal among conformal procedures with equal 
or greater constant coverage, because all constant coverage conformal procedures
have constant conditional coverage given $Z$. 

Finally, we note that just because the model $\mathcal P$ is 
nonparametric, the prior distribution $\pi$ does not need to be nonparametric 
for the conditions of Theorem \ref{thm:nppred} to hold. For example, 
if $\mathcal Y$ is a Euclidean space and $\mu$ is Lebesgue measure, 
then a prior distribution $\pi$ that has mass only on the normal 
distributions yields a $P_\pi^X$ that satisfies 
assumption A2 of the theorem. A $1-\alpha$ prediction region constructed using such 
a $P_\pi^X$ still has exact $1-\alpha$ coverage under \emph{all}
probability distributions dominated by Lebesgue measure, 
and is expected to have a small volume if the true distribution 
$P$ is a normal distribution, 
but might have a large volume 
if $P$ is far from normality.

\subsection{Normal populations} 
Let $X\sim N_p(\theta, k \Sigma)$ and 
$Y\sim N_p(\theta,\Sigma)$ be independent. 
We first consider predicting $Y$ from $X$ in the case that 
$\Sigma$ is known. In this case $Z=(X+kY)/(1+k)$ is a 
complete sufficient statistic, with $P_z^X$ being the 
$N_p(z, \Sigma k^2/(k+1))$ distribution. 
Since $Z_x$ is injective for each $x$, a Bayes-optimal 
$1-\alpha$ prediction region can be constructed 
from Equation \ref{eqn:dibopt}, so in particular, 
a value of $y$ is accepted into the prediction region 
if $p_z^X(x)/p_\pi^X(x)$ exceeds the  $\alpha$ quantile  
of $p_z^X(X)/p_\pi^X(X)$ under $X\sim P^X_z$, where 
$z=Z(x,y)$. 
If the prior information for $\theta$ is represented by 
$\theta\sim N_p(\mu,\lambda \Sigma)$ 
for some $\mu\in \mathbb R^p$ and $\lambda>0$, 
then 
$P_\pi^X$ is the $N_p(\mu, (k+\lambda)\Sigma)$ distribution. 
After some manipulation of $\log  p_z^X(x)/p_\pi^X(x)$, 
a  Bayes-optimal test statistic is obtained:
\begin{align}
t_z(x) & =\lVert  \Sigma^{-1/2}(x-z)/v^{1/2}  +  \delta_z \rVert^2
\label{eqn:nzxstat} 
\end{align}
where $\delta_z= \Sigma^{-1/2}(\mu -z) v^{1/2}/(v_\lambda - v)$ with 
$v = k^2/(k+1)$ and $v_\lambda= (k+\lambda)$, and 
$\Sigma^{1/2}$ is any matrix for which 
$\Sigma^{1/2}(\Sigma^{1/2})^\top = \Sigma$. 
Since $\Sigma^{-1/2}(x-z)/v^{1/2}
 \sim N(0,I)$ 
under $H_z: X \sim P^X_z$,
the null distribution of the statistic 
is a non-central $\chi^2_p$ distribution with noncentrality 
parameter $ \lVert \delta_z\rVert^2$. Thus, a value $y$ is included in the 
prediction region if $t_{Z(x,y)}(x)$ is less than
$\chi^2_{p,\lVert \delta_z\rVert^2,1-\alpha}$, 
the 
$1-\alpha$ quantile of this distribution. 

Some intuition for this statistic may be obtained by expressing 
it slightly differently. 
One useful re\"expression is obtained by recalling that
$z=(x+ky)/(1+k)$ on $Z(x,y)=z$, which gives 
\begin{align}
t_{Z(x,y)}(x) = \lVert\Sigma^{-1/2} (x-y)/\sqrt{k+1} + \delta_{Z(x,y)}\rVert^2,
\label{eqn:npstatF}
\end{align}
and so 
the Bayes-optimal $1-\alpha$ prediction region can be written
\begin{equation}
 A_x^\pi = \{ y : \lVert  \Sigma^{-1/2}(x - y)/\sqrt{k+1} + \delta_{Z(x,y)} \rVert^2
 < \chi^2_{p,\lVert \delta_{Z(x,y)}\rVert^2,1-\alpha} \}. 
\label{eqn:npregionF}
\end{equation}
For comparison, the usual equivariant prediction region is 
\begin{equation} 
 A_x^E= \{ y : \lVert \Sigma^{-1/2} (x-y)/\sqrt{k+1} \rVert^2 < \chi^2_{p,0,1-\alpha}\}. 
\end{equation}
Since  $\delta_z \rightarrow  0$ as $\lambda\rightarrow \infty$, this 
standard region can be viewed as a limit 
of regions of the form (\ref{eqn:npregionF}) 
under a sequence of 
increasingly non-informative prior 
distributions. 

The optimal statistic (\ref{eqn:nzxstat}) and region (\ref{eqn:npregionF}) 
can also be related to 
a fully Bayesian prediction procedure as well. 
Rearranging terms in (\ref{eqn:nzxstat})
gives 
\begin{align} 
t_{Z(x,y)} &= \lVert  \Sigma^{-1/2}( y - \hat\theta^\pi)/v_\lambda^{1/2}    \rVert^2 \times
 (k+1)/v_\lambda \label{eqn:nzystat} \\
A_x^\pi & = 
\{ y : \lVert  \Sigma^{-1/2}( y - \hat\theta^\pi)/v_\lambda^{1/2}    \rVert^2 \times
 (k+1)/v_\lambda < \chi^2_{p,\lVert\delta_{Z(x,y)}\rVert^2,1-\alpha} \}, 
\label{eqn:npregion}
\end{align}
where $\hat\theta^\pi = (X/k + \mu/\lambda)/(1/k + 1/\lambda)$ is 
the posterior mean estimator of $\theta$ given $X$, and 
$v_\lambda = (\lambda(k+1)+k)/(k+\lambda)$ 
relates to the 
prior predictive variance of 
$Y- \hat\theta$,  which is $v_\lambda \Sigma$. 
For comparison, 
the posterior predictive distribution for $Y$ under the prior 
$\theta\sim N_p(\mu,\lambda I)$ is 
$Y|\{X=x\} \sim N_p( \hat\theta^\pi , v_\lambda\Sigma)$, and so the
fully Bayesian 
$1-\alpha$ prediction region with highest posterior predictive density 
is given by
\begin{equation} A_x^B = \{ y : \lVert  \Sigma^{-1/2}( y - \hat\theta^\pi)/v_\lambda^{1/2}    \rVert^2 
 < \chi^2_{p,0,1-\alpha} \}. 
\end{equation}
However, since the 
distribution of $\lVert\Sigma^{-1/2}( y - \hat\theta^\pi)/v_\lambda^{1/2} \rVert^2$
(unconditionally on $Z$) 
depends on the unknown value of $\theta$, this fully Bayesian posterior predictive region will have frequentist coverage that varies as a function of $\theta$. 
In contrast, the 
FAB region is obtained by 
comparing (\ref{eqn:nzystat}) to its $1-\alpha$ conditional quantile 
given $Z=z$ in order to maintain constant frequentist coverage as a 
function of $\theta$. In particular, unlike the posterior predictive region, 
the region (\ref{eqn:npregion}) 
 is \emph{not} centered around 
$\hat\theta^\pi$ because the conditional quantile depends on the noncentrality parameter $\lVert\delta_z\rVert^2$, which varies as a function of $z=Z(x,y)$ and hence 
varies with $y$.

Some numerical comparisons of FAB regions to the standard equivariant regions
are given in Figures \ref{fig:np1}, \ref{fig:np2} and \ref{fig:nrisk}. 
The first figure  displays 90\% prediction intervals and widths 
as a function of $x$ for the case that $p=1$, $\sigma^2=1$, $k=1$ and 
$\mu=0$. 
As shown in the left panel for the case that $\lambda=1$, 
when $x$ is close to zero (as is expected under the prior distribution), the FAB interval is narrower than the equivariant interval 
($\lambda = \infty)$, at the cost of being wider for 
values of $x$ that are less likely under the prior distribution. 
The right panel of the figure summarizes the widths for a range of $\lambda$ 
values. The biggest 
differences between the FAB and equivariant intervals 
occur for highly informative prior distributions, that is, 
when $\lambda$ is small. In contrast, for large values of $\lambda$ the FAB interval 
can be better than the usual interval over a wide range of $x$ values, 
but the improvement is smaller. 
Figure \ref{fig:np2} displays analogous results for the case 
that $p=2$, $\Sigma=I$, $k=1$ and $\mu=(0,0)$. 
When $x$ is in accord with the prior distribution the 
FAB region can be substantially smaller than the usual region
 - close to half the area in this case. 
As $x$ moves away from $\mu$ the area increases 
in order to accommodate 
both the prior distribution and the requirement of  
90\% frequentist coverage. 
Figure \ref{fig:nrisk} compares the frequentist risk of the 
FAB procedures as a function of $\theta$ and $\lambda$ for $p=1$ and $p=2$. 
The risk differences are smaller than the volume differences
as the former are obtained by averaging the latter over the values of $X$, 
with respect to the distribution $X\sim N_p(\theta,k\Sigma)$.

\begin{figure}[ht]
\centering{\includegraphics[height=2.75in]{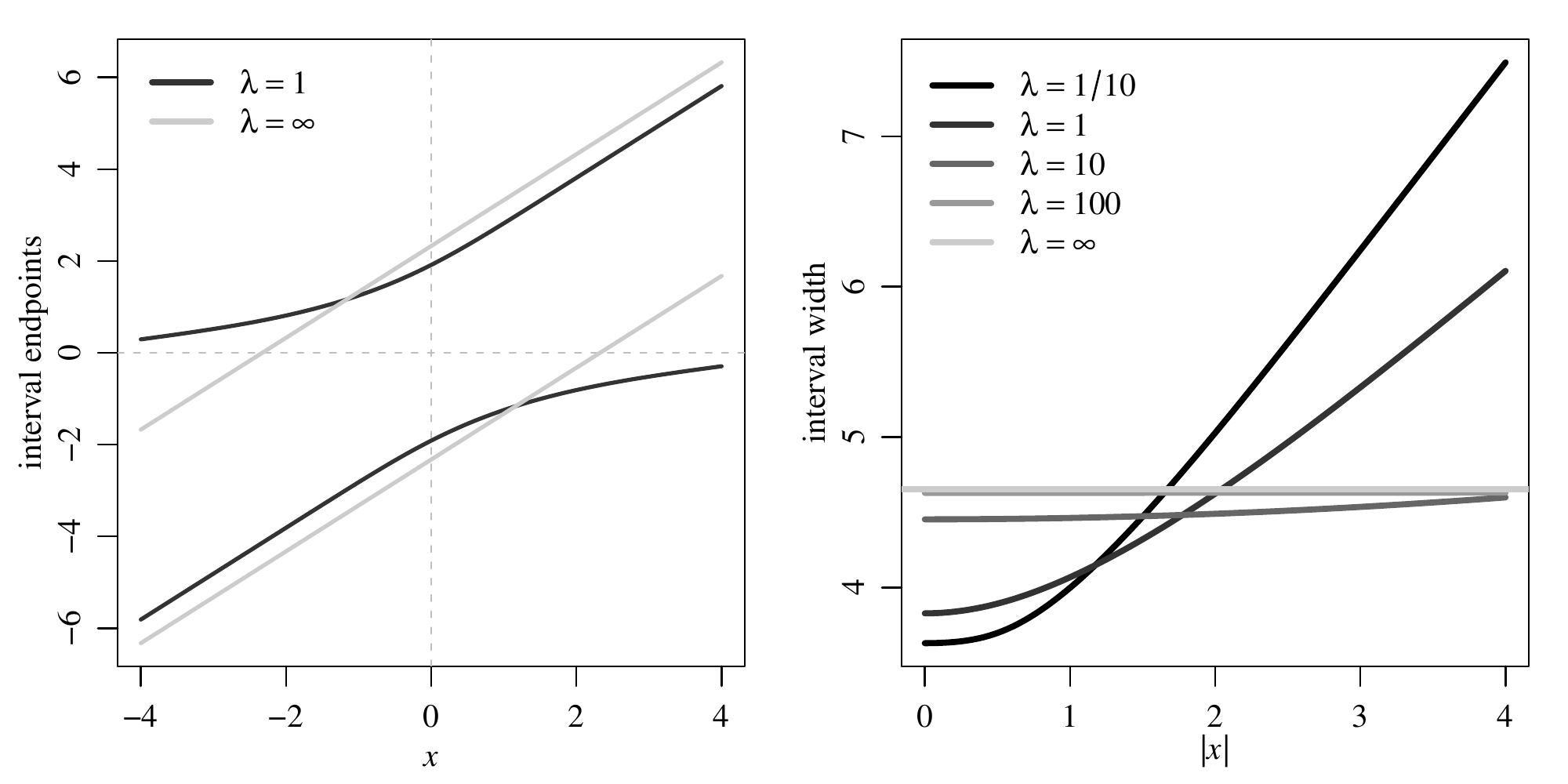}} 
\caption{Normal 90\% prediction intervals for $p=1$. Left panel: Intervals as a function of $x$ for $\lambda=1$ 
and $\lambda=\infty$. 
 Right panel: Interval widths as a function of $|x|$ for $\lambda\in \{ .1,1,10,100,\infty\}$. } 
\label{fig:np1}
\end{figure}

\begin{figure}[ht]
\centering{\includegraphics[height=2.75in]{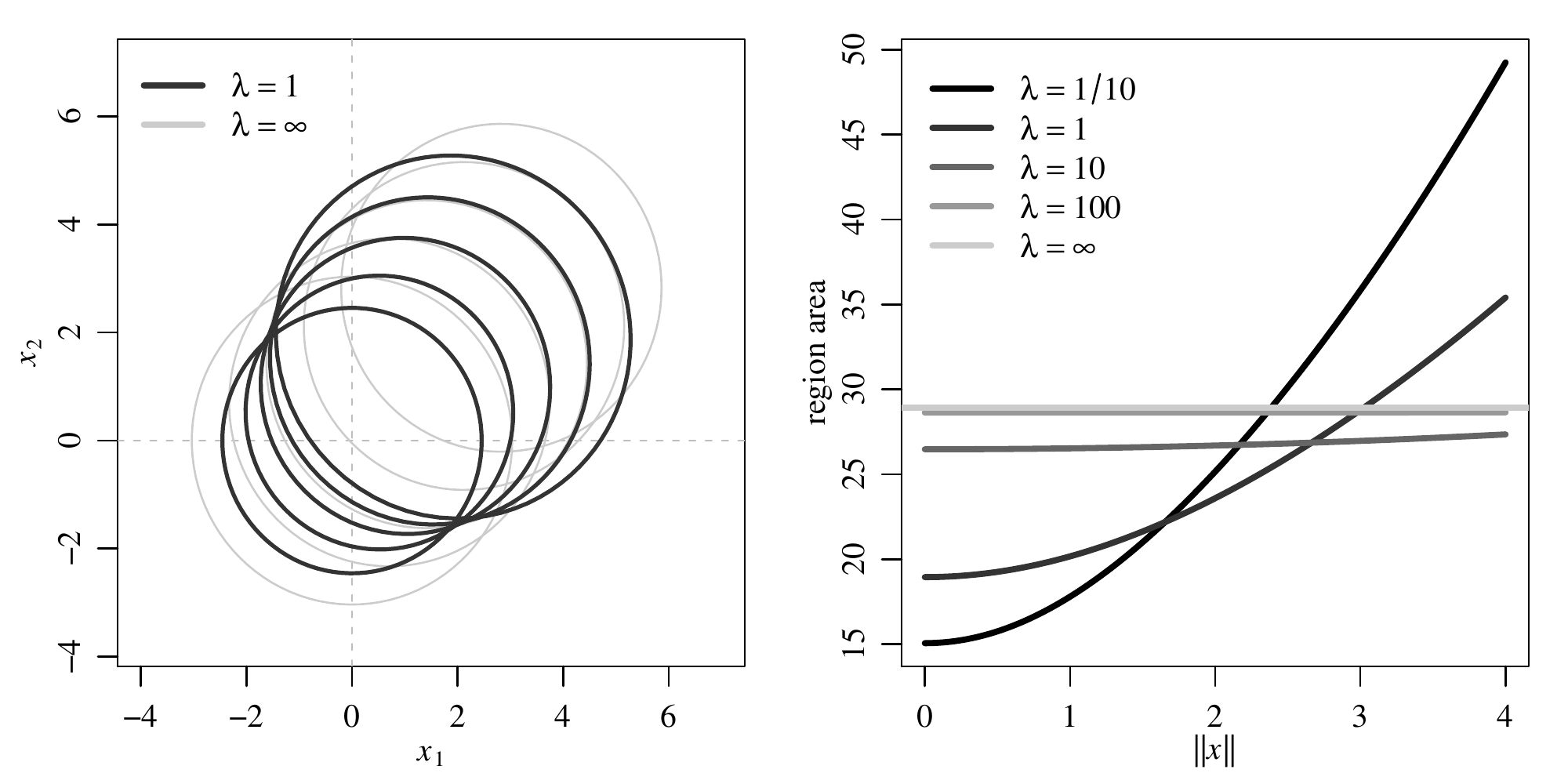}}
\caption{Normal 90\% prediction regions for $p=2$. Left panel: Regions as a function of $x$ for $\lambda=1$ 
and $\lambda=\infty$. 
 Right panel: Region areas as a function of $\lVert x\rVert$ for $\lambda\in \{ .1,1,10,100,\infty\}$. } 
\label{fig:np2}
\end{figure}

\begin{figure}[ht]
\centering{\includegraphics[height=2.75in]{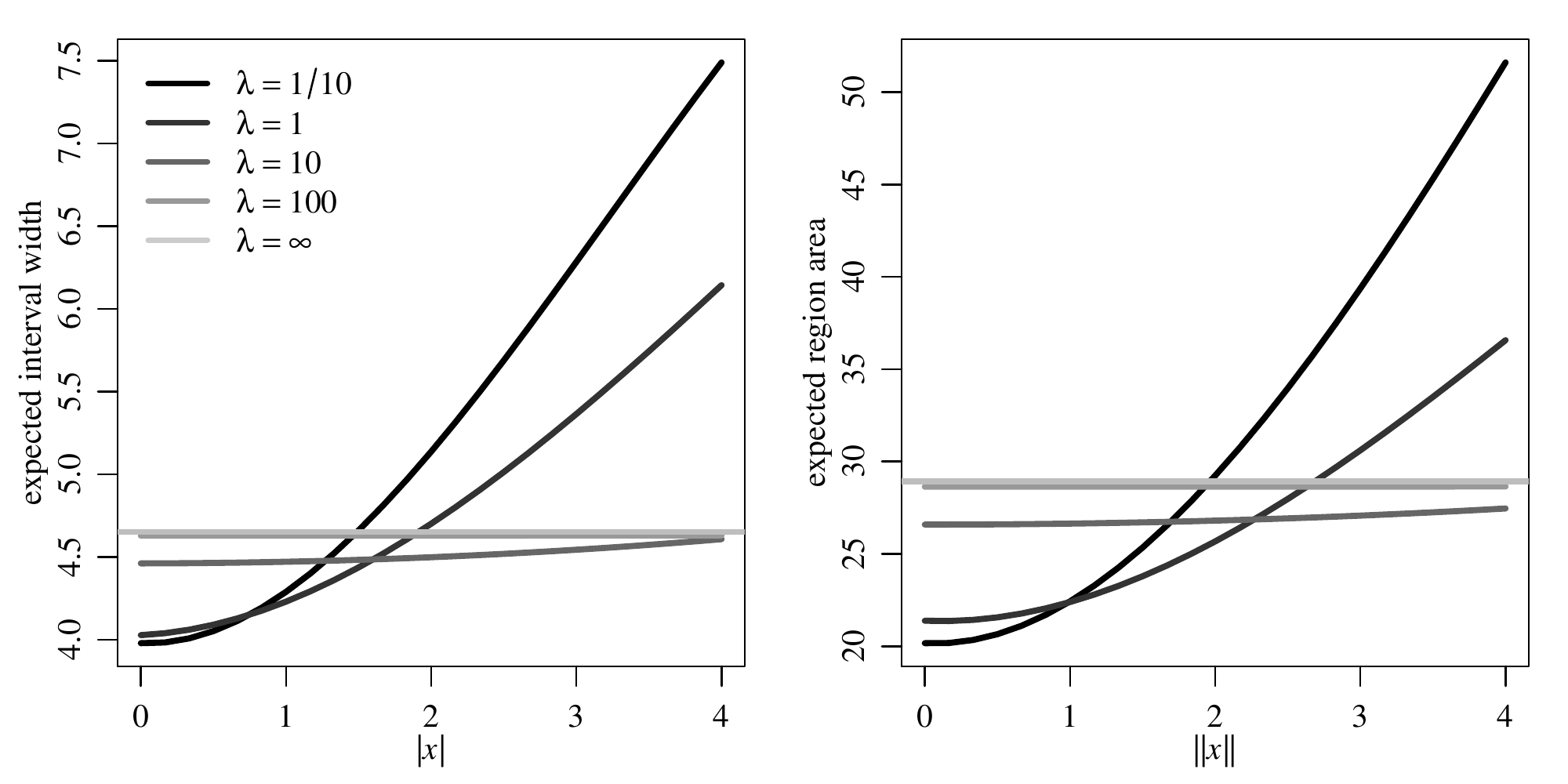}}
\caption{
Expected 90\% prediction region volumes as a function of $x$ and $\lambda$ for 
$p=1$ (left panel) and $p=2$ (right panel).  }
\label{fig:nrisk}
\end{figure}

In practice the covariance matrix $\Sigma$ will be unknown, but estimable 
from available data. For example, suppose we want to predict 
$Y\sim N_p(\theta,\Sigma)$ 
from $X=(\hat\theta, \hat\Sigma)$, where 
$\hat\theta\sim N_p(\theta,k\Sigma)$ and 
$\nu \hat\Sigma\sim$Wishart$(\nu,\Sigma)$
are independent of each other. 
A complete sufficient statistic for the joint distribution 
of $Y$ and $X$ is $Z=\{ (\hat\theta+kY)/(1+k) , \nu\hat \Sigma + YY^\top + 
   \hat\theta \hat\theta^\top/k$\}. 
Furthermore, $Z_x$ is injective for each $x$, and so 
given a prior distribution on $(\theta,\Sigma)$ one could construct 
a FAB prediction region from 
(\ref{eqn:dibopt}). However, without going into too many details, this 
approach will be quite cumbersome as it involves singular conditional 
distributions and an optimal test statistic that 
must be numerically approximated, 
at least for any prior distribution on $\Sigma$ of which I am aware. 
As an alternative, 
a simpler FAB statistic may be constructed by replacing each appearance of 
$\Sigma$ in (\ref{eqn:npstatF}) with estimates: Changing notation 
slightly, let $\hat\theta=X$, 
$Z=( X+kY )/(1+k)$ and consider the statistic
\begin{equation}
t_{Z(x,y)}(x)= \lVert \hat\Sigma^{-1/2} (x-y)/\sqrt{k+1} + \tilde  \delta_{Z(x,y)}   \rVert^2
\end{equation}
where $\tilde \delta_z = \tilde \Sigma^{-1/2}(\mu-z)v^{1/2}/(v_\lambda-v)$, 
and $\hat\Sigma^{-1/2}$ and $\tilde \Sigma^{-1/2}$ 
are estimates of 
of $\Sigma^{-1/2}$. Specifically, 
let $\hat \Sigma^{-1/2}$ be the Cholesky factorization of $\hat\Sigma^{-1}$, 
and let $\tilde \Sigma^{-1/2}$ 
be any other estimate of $\Sigma^{-1/2}$ that is deterministic (e.g.\ based on prior information) or statistically independent 
of $\hat\theta$ and $\hat\Sigma$. 
It is straightforward to show that the 
conditional distribution of $(X-Y)/\sqrt{k+1}$ given $Z=z$ is
$N_p(0,\Sigma)$, and further that the conditional distribution of 
$\hat\Sigma^{-1/2}(X-Y)/\sqrt{k+1}$ given $Z=z$ 
does not depend on the unknown 
parameters $(\mu, \Sigma)$.
Therefore, $t_{Z(x,y)}$ can be used 
to evaluate each hypothesis  $H_z:X \sim P^X_z$, and these tests may 
be inverted to obtain the approximately optimal prediction 
region 
\begin{equation}
\tilde A_x^\pi= \{ y : \lVert \hat\Sigma^{-1/2} (x-y)/\sqrt{k+1} + \tilde \delta_{Z(x,y)} \rVert^2 < q_{\tilde \delta_{Z(x,y)},1-\alpha}  \}
\label{npredT} 
\end{equation}
where $q_{b,1-\alpha}$ is the 
$1-\alpha$ quantile of the 
distribution of $\lVert T +  b \rVert^2$, 
where $T\stackrel{d}{=}
\hat\Sigma^{-1/2}(X-Y)/\sqrt{k+1}$. 
If $p=1$ then $T$ has the $t$-distribution with $\nu$ degrees of freedom.

\subsection{Linear regression}
Suppose we wish to predict $Y \sim N(v^\top \beta , \sigma^2)$ 
from $X \sim N_n( U \beta ,\sigma^2 I )$ where 
$X$ and $Y$ are independent, and $v\in \mathbb R^p$ and $U\in \mathbb R^{n\times p}$ are fixed, non-stochastic explanatory variables. We first consider the case that $\sigma^2$ is known. In this case, $Z = U^\top X + v Y$ is a complete sufficient statistic, and the joint density of $(X,Y)$ factorizes as 
$p_\beta(x,y) = h_\beta(z) \times g(x,y)$ where 
 $g(x,y) = \exp\{ -( x^\top x + y^2 )/(2\sigma^2) \}$. 
Under the prior distribution 
$\beta\sim N_p(0,\sigma^2 \Psi^{-1})$, 
the prior predictive distribution of $X$ is 
$N_p( 0, \sigma^2 ( I +  U\Psi^{-1} U^\top ) )$.
By Corollary \ref{cor:nfrep}, the FAB prediction region is formed by 
inverting tests 
that accept $H_z: (X,Y)\sim P_z$  when $g(x,y)/p^X_\pi(x)$ is large. 
After some manipulation, we have 
\begin{equation}
\label{eqn:reglrt}
-2 \sigma^2 \log  g(x,y)/p^X_\pi(x) = 
y^2 + x^\top U (  U^\top U + \Psi)^{-1} U^\top x. 
\end{equation}
On $Z(x,y)=z$ 
we have $U^\top x = z - v y$ and so 
(\ref{eqn:reglrt})  can be written 
\begin{equation}
-2 \sigma^2  \log  g(x,y)/p^X_\pi(x) =  
y^2 [ 1+ v^\top ( U^\top U+ \Psi )^{-1} v ] - 
  2 y v^\top (U^\top U  + \Psi )^{-1} z+ c_z
\end{equation}
where $c_z$ does not depend on $y$. By completing the square and rearranging, we see that an optimal 
test is one that accepts $H_z$ for small values of the 
statistic 
\begin{align} 
t_z(y) & =  
\tfrac{1}{\sigma}\sqrt{1+v^\top (U^\top U)^{-1} v}  \times 
  \left \lvert  y - \frac{ v^\top ( U^\top U+\Psi)^{-1}z }{ 1+ v^\top (U^\top U + \Psi)^{-1}v}  \right \rvert  \nonumber  \\
& =  \tfrac{\sqrt{w_0}}{\sigma} 
  \lvert 
y - v^\top [S_\psi/w_\psi ] z  
\rvert 
\label{eqn:fabregstat0}
\end{align}  
where $S_\psi = (U^\top U + \Psi)^{-1}$ and
$w_\psi = 1+v^\top S_\psi v$, and so 
in particular $S_0=(U^\top U)^{-1}$ and 
 $w_0 = 1+v^\top (U^\top U)^{-1} v$. 
To find the critical value for this test, 
rewrite (\ref{eqn:fabregstat0}) as
\begin{equation}
 t_z(y) = \left  \lvert \tfrac{\sqrt{w_0}}{\sigma} (y-v^\top [ S_0/w_0]   z) + \delta_z \right \rvert, 
\label{eqn:fabregstat}
\end{equation} 
where $\delta_z = v^\top(S_0/w_0 -S_\psi/w_\psi)z \sqrt{w_0}/\sigma$. 
Since
$Y|\{Z=z\} \sim N( v^\top  [S_0/w_0]z,\sigma^2/w_0)$, 
the critical value $q_z$ for the 
level-$\alpha$ test using this statistic is the solution in $q$
to the equation 
 $\Phi( q- \delta_z  ) -
 \Phi( -q -\delta_z  )   = 1-\alpha$, which may be obtained numerically. 
The $1-\alpha$ FAB prediction region 
therefore consists of 
$y$-values for which $t_{Z(x,y)}(y) \leq q_{Z(x,y)}$. 

The prediction region generated by the FAB test statistic (\ref{eqn:fabregstat}) bears some resemblance to 
the standard equivariant $1-\alpha$ prediction region:
The first term in   (\ref{eqn:fabregstat}) 
may be rearranged as follows:
\begin{align}
 y- v^\top[ S_0/w_0]z  &= 
 y- v^\top (U^\top U)^{-1}(U^\top x + v y )/w_0  \\ 
&= y- v^\top \hat\beta/w_0 - (w_0-1)y/w_0 \\
& = (y-\hat\beta^\top v )/w_0
\end{align}
where $\hat \beta$ is the ordinary least-squares estimate. 
Therefore, the statistic may be written 
$t_z(y) = \lvert (y-\hat \beta^\top v)/(\sigma \sqrt{w_0}) +\delta_{z}\rvert$,
and the FAB prediction region is 
\begin{equation}
A^\pi_x = \{ y : \lvert (y-\hat \beta^\top v)/(\sigma \sqrt{w_0}) +\delta_{Z(x,y)}\rvert  \leq q_{Z(x,y)} \}. 
\label{eqn:fabregols} 
\end{equation}
For comparison, the equivariant region 
based on the pivotal quantity $y-\hat \beta^\top v$
is 
$\{ y : \lvert (y-\hat \beta^\top v)/(\sigma\sqrt{w_0}) \rvert \leq 
   \Phi^{-1}(1-\alpha/2) \}$. 
This interval is a limit of FAB intervals of the 
form (\ref{eqn:fabregols}):
As the prior distribution becomes 
increasingly non-informative 
as the eigenvalues of $\Psi$ go to zero, 
we have 
$S_\psi \rightarrow S_0$ and $w_\psi \rightarrow w_0$, 
and so $\delta_z \rightarrow 0$ and 
$q_z \rightarrow \Phi^{-1}(1-\alpha/2)$.

To relate 
the FAB region to the 
fully Bayesian region derived from the posterior predictive distribution 
of $Y$ given $X=x$, note that 
\begin{align}
 v^\top (U^\top U + \Psi  )^{-1} z &= 
 v^\top (U^\top U +\Psi )^{-1}( U^\top x + v y ) \\
 &=  v^\top \hat\beta_\psi  + (w_\psi-1) y
\end{align}
where $\hat\beta_\psi = (U^\top U + \Psi)^{-1}U^\top x$ is the 
posterior mean of $\beta$ given $X=x$ 
under the prior distribution $\beta\sim N_p(0,\sigma^2 \Psi^{-1})$. 
The statistic (\ref{eqn:fabregstat}) may therefore be 
written 
\begin{equation}
t_{Z(x,y)} =  \lvert y- \hat\beta_\psi^\top v \rvert \times \sqrt{w_0}/(\sigma w_\psi ), 
\end{equation}
and the prediction region may be
expressed as 
\begin{equation}
A^\pi_x = \{ y : \hat\beta_\psi^\top v  - q_{Z(x,y)} \sigma w_\psi/\sqrt{w_0} < y <
 \hat\beta_{\psi}^\top v + q_{Z(x,y)} \sigma w_\psi/\sqrt{w_0} \}.  
\label{eqn:fabregridge} 
\end{equation}
This can be compared to the fully Bayesian interval as follows: 
The posterior distribution of $\beta$ given $X=x$ is 
$\beta|\{X=x\}
 \sim N_p( \hat\beta_\psi , \sigma^2(U^\top U + \Psi)^{-1})$, which 
gives the posterior predictive distribution 
$Y|\{X=x\} \sim N ( \hat\beta_\psi^\top v , 
            \sigma^2 w_\psi )$. 
The 
$1-\alpha$ highest posterior predictive density region
for $Y$ given $X=x$ is therefore
\begin{equation}
A_x= \{ y:  
   \hat\beta_\psi^\top v  - \Phi^{-1}(1-\alpha/2) \sigma \sqrt{w_\psi} < y <
 \hat\beta_{\psi}^\top v + \Phi^{-1}(1-\alpha/2)  \sigma \sqrt{w_\psi}\}.  
\label{eqn:bayesreg} 
\end{equation}
This posterior predictive interval is centered around the biased estimator $\hat\beta_\psi^\top v$ of $\beta^\top v$, and 
as a result will have a frequentist coverage rate that varies in $\beta$. 
In contrast, the FAB interval (\ref{eqn:fabregridge}) is 
not centered around $\hat\beta_\psi^\top v$, as the 
acceptable deviation from this estimator varies in  $y$ through 
$q_{Z(x,y)}$, which is set conditionally in order to maintain $1-\alpha$ frequentist coverage 
for all values of $\beta$.

Typically the value of $\sigma^2$ is unknown and must be estimated from the 
data. As in the previous subsection, we forego formulating a 
prior distribution for $\sigma^2$  and instead describe 
a FAB prediction region that maintains exact $1-\alpha$ frequentist 
coverage for every value of $\beta$ but is only approximately risk-optimal.
As an alternative to (\ref{eqn:fabregstat}) 
consider the statistic 
\begin{align}
t_z(y) = \lvert \tfrac{\sqrt{w_0}}{\hat \sigma}(y-v^\top [S_0/w_0]z)+\tilde \delta_z\rvert , 
\label{eqn:fabregtstat} 
\end{align}
where
$\tilde \delta_z = v^\top(S_0/w_0 -S_\psi/w_\psi)z \sqrt{w_0}/\tilde\sigma$, 
with 
$\hat\sigma^2$ 
and $\tilde \sigma^2$ being two estimates of 
$\sigma^2$ that are statistically independent of each other and 
of $U^\top X$. 
In particular, assume that $\nu \hat\sigma^2/\sigma^2\sim \chi^2_\nu$, 
in which case
$(y-v^\top [S_0/w_0]z) \sqrt{w_0}/\hat\sigma\sim t_{\nu}$ under 
$H_z: Y\sim P^Y_z$. The critical value $q_z$ for a level-$\alpha$ 
test of $H_z$
therefore satisfies
$ F_\nu( q_z- \tilde \delta_z ) - F_\nu(-q_z- \tilde \delta_z )  = 1-\alpha$,
where $F_\nu$ is the cumulative distribution function of the $t_\nu$ distribution. Rearranging terms as before, the resulting FAB prediction region can be expressed 
\begin{equation}
A^\pi_x = \{ y : \lvert (y-\hat \beta^\top v)/(\hat\sigma\sqrt{w_0}) +\tilde 
\delta_{z}\rvert  \leq q_{Z(x,y)} \}. 
\label{eqn:fabregolsT} 
\end{equation}
Appropriate values of 
$\hat\sigma^2$ and $\tilde \sigma^2$ can often be obtained 
from the data at hand: For example, 
both of these quantities may be obtained 
from a decomposition of the 
residual sum of squares from the regression of $X$ on $U$. Details on 
this approach 
are available from the replication files for this article at my website. 
Alternatively, $\hat\sigma^2$ could be the usual unbiased 
estimate of $\sigma^2$ from this regression, with $\tilde \sigma^2$ coming from 
non-stochastic prior information or other data that are independent of $X$. 

\begin{figure}[ht]
\centering{\includegraphics[height=2.5in]{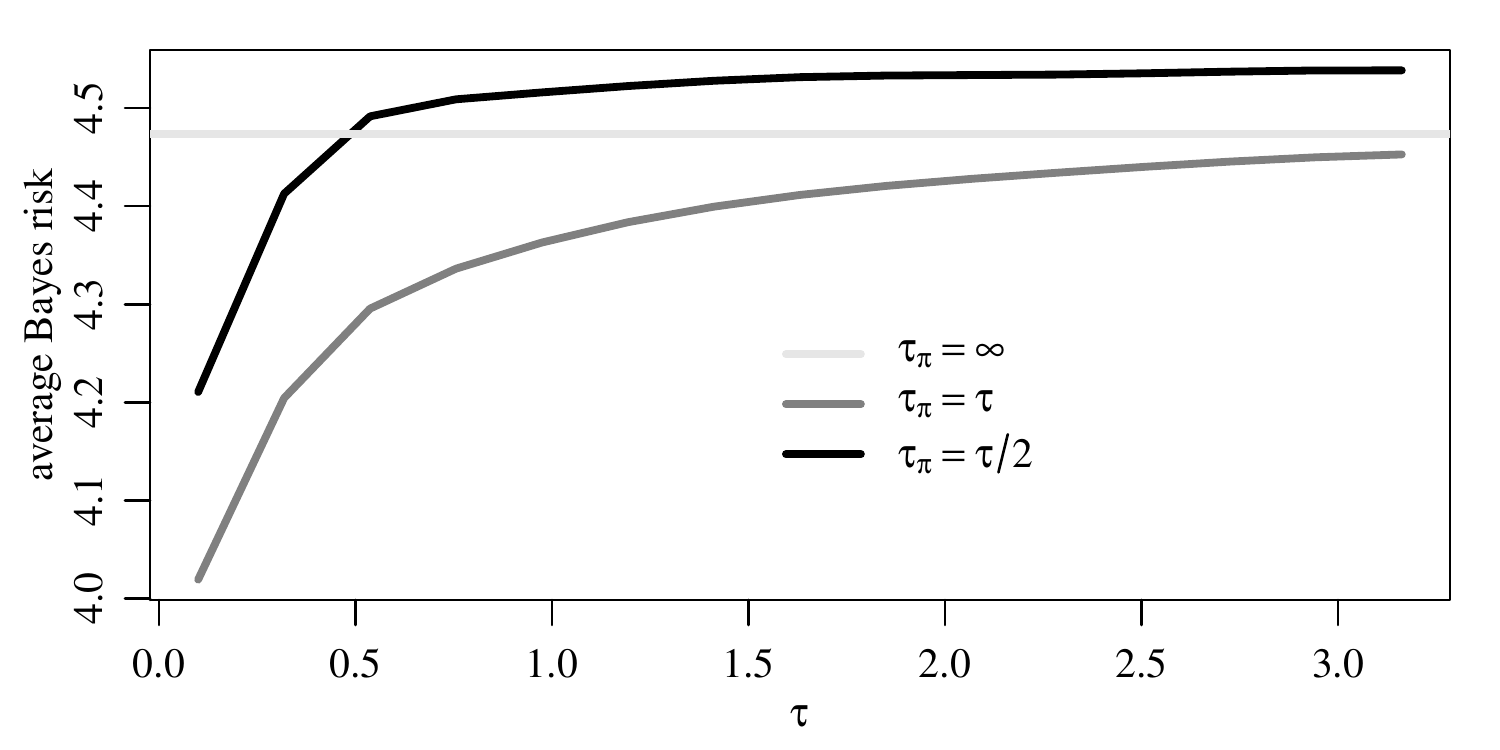}} 
\caption{Bayes risk comparisons of 90\% FAB and equivariant prediction intervals
for the normal linear model,
averaged across vectors of explanatory variables. The horizontal gray line is the risk of the equivariant interval. The medium gray line is the average Bayes risk of the FAB interval. The black line is the average Bayes risk of a FAB interval using an overly-concentrated prior distribution. } 
\label{fig:lmrisk} 
\end{figure}

Some risk comparisons are displayed in Figure \ref{fig:lmrisk}. A single 
$100\times 75$ matrix $U$ was randomly generated, with columns that 
were correlated but having 
zero mean and unit variance. 
Expected widths of regions for predicting $Y\sim N(v^\top \beta,\sigma^2)$ 
were computed, where $v$ ranged over the rows of $U$. The expected interval widths 
were averaged over these values of $v$ to obtain an average Bayes risk. 
This was done for $\sigma^2=1$ and 
under prior distributions
$\beta\sim N_{75}(0,\tau^2 I)$, where 
$\tau^2$ ranged from 1/10 to 10. 
The horizontal light gray line in the figure 
gives the 
average 
width of the standard equivariant interval
(labeled with $\tau_\pi=\infty)$, 
which 
is constant as a function of $\tau^2$. 
The medium gray line gives the average Bayes risk of the 
approximately optimal FAB procedure, given by (\ref{eqn:fabregolsT}), 
 as a function of $\tau$ 
(labeled $\tau_\pi=\tau$). For the lowest value $\tau^2=1/10$, 
the FAB interval has an average expected width that is about 10\% smaller than 
the equivariant interval. As $\tau^2$ increases 
the risk remains below that of the equivariant interval but 
the improvement decreases,  
in accord with our understanding of the equivariant interval as a limit 
of FAB intervals.  

We briefly consider the performance of the FAB procedure when the 
prior distribution is not in accord with the actual value of 
$\beta$.  
We expect that the FAB procedure will improve upon the equivariant interval 
as long as the prior distribution $\pi$ 
is sufficiently diffuse, i.e.\ the prior variance $\tau^2$ 
used to construct the FAB region is not much smaller than the true
magnitude of $\beta$. However, using a very diffuse prior will 
not provide much of a risk improvement over the equivariant interval.
Since the risk improvements are achieved when $\tau^2$ is small, 
it is more useful to consider the potential downside
to using a  prior distribution that is overly concentrated around zero. 
An example of this behavior is given by the black line in 
Figure \ref{fig:lmrisk}, which gives the Bayes risk 
of a FAB procedure for which the 
``true'' $\tau^2$ (the one used to compute the Bayes risk) is 
is four times that of $\tau_\pi^2$, the value used to construct 
the FAB region. The figure indicates that, for these values of $U$ and 
$\sigma^2$, an overly-concentrated prior distribution still yields 
a FAB region that improves upon the equivariant interval for small values 
of $\tau^2$, but can be worse for larger values. However, we reiterate that even if the prior distribution is misspecified, the the FAB procedure 
will still maintain $1-\alpha$ frequentist coverage for each value of $\beta$, no matter how large or small.

\section{Discussion}
The FAB prediction procedure introduced in this article allows for 
incorporation 
of prior or indirect information
while still maintaining a target frequentist coverage rate.
In many cases, the FAB procedure is Bayes-risk optimal among procedures 
with a given frequentist coverage rate.
In practice, 
this means that a FAB prediction region will have a smaller expected volume 
than other regions with the same coverage rate, 
if there is not a large discrepancy between the prior distribution and 
the population from which the data are to be sampled. 
If there is a large discrepancy, the FAB region 
will still maintain the target 
frequentist coverage rate but could have 
a volume that is large compared to other procedures. 
This raises the question of when, in practice, would one be confident enough in their prior distribution to use a FAB procedure? One possibility is with
 multipopulation scenarios, where the ``prior distribution'' for a
population may be obtained using data from the others, perhaps via a hierarchical model. Specifically, suppose $(X_j,Y_j)\sim P_{\theta_j}$ 
independently for $j\in\{1,\ldots, p\}$. For prediction of 
$Y_j$ from $X_j$, one could fit a hierarchical model for the 
parameters $\theta_1,\ldots, \theta_p$ using data
$X_1,\ldots, X_{j-1},X_{j+1},\ldots, X_p$. The hierarchical model 
provides indirect information about $\theta_j$ that can be used 
to construct a FAB prediction region.
The resulting region for $Y_j$ will still maintain 
frequentist coverage  because the prior distribution for 
$\theta_j$ is obtained
from data that are 
statistically independent 
of $(Y_j,X_j)$. This type of approach has been used before
for frequentist confidence interval construction \citep{yu_hoff_2018,burris_hoff_2018}. In those applications, FAB confidence intervals 
maintain population-specific frequentist coverage guarantees, while being 
narrower than standard frequentist procedures on average across populations. 

In addition to the aforementioned application to multipopulation inference, 
other areas of further research include development of methods for specific models, such as  general and generalized linear models, exponentially parameterized random graph models for networks, and 
the contamination models described in \citet{mattner_1996}. 
Nontrivial details to be worked out for each model include 
identification of optimal test statistics as well as 
efficient computational 
methods for inverting the tests to construct a prediction region. 
Additionally, in models for which 
the conditions of Theorem \ref{thm:boptlevel} do not hold, 
there remains the interesting 
open question of whether or not 
a FAB prediction region can have larger Bayes risk than a region 
with equal or greater coverage. 

\smallskip

Computer code to construct FAB prediction regions for the multivariate 
normal and normal linear regression models is available at 
\url{https://github.com/pdhoff/FABPrediction}.

\appendix 

\section*{Proofs}

\begin{proof}[Proof of Lemma \ref{lem:svfr}]  
First we show that 
the function that maps each 
set-valued function to its graphs is a bijection. 
To see that this function is surjective, note that 
for any subset $A$ of $\mathcal X\times \mathcal Y$, the 
graph of the set-valued function $x\mapsto A_x$, 
where $A_x=\{ y:(x,y)\in A\}$, is $A$. 
To see that it is injective, suppose that 
two set-valued functions, say $x\mapsto A_x$ and 
$x\mapsto A_x'$, are not the same. Then there exists 
an $\tilde x$ for which $A_{\tilde x} \neq A_{\tilde x}'$, and so there is a 
$\tilde y$ that is an element of one but not the other. Suppose 
$\tilde y$ is in $A_{\tilde x}$ but is not in $A_{\tilde x}'$. 
Then $(\tilde x,\tilde y)$ is in the 
graph of $x\mapsto A_x$ but not that of 
$x\mapsto A_x'$. 

Now we show that the function that  maps a set-valued 
function $z\mapsto A_z$ to $\cup_{z\in \mathcal Z} A_z$ is 
a bijection from item 3 to item 1. To see that it is surjective, 
note that for a given subset  $A$ of $\mathcal X\times\mathcal Y$ 
the set-valued function $z\mapsto A_z$ with $A_z= A \cap Z^{-1}\{z\}$ satisfies
$\cup_{z\in \mathcal Z} A_z = A$. To see that it is injective, 
suppose the functions 
$z\mapsto A_z$ and
$z\mapsto A_z'$, are not the same, 
and that $A_z$ and $A_z'$ are subsets of $Z^{-1}\{z\}$ for each $z$.  
Then there exists
a $\tilde z$ for which $A_{\tilde z} \neq A_{\tilde z}'$, 
and so there is a point
$(\tilde x,\tilde y)$ 
with $\tilde z = Z(\tilde x, \tilde y)$ 
that is an element of one but not the other. Suppose
$(\tilde x,\tilde y)$ is in $A_{\tilde z}$ but is not in $A_{\tilde z}'$. Then $(\tilde x,\tilde y)$ must be in 
$\cup_{z\in \mathcal Z} A_z$ because it is in $A_{\tilde z}$. 
But $(\tilde x,\tilde y)$ cannot be in 
$\cup_{z\in \mathcal Z} A_z'$, because 
it is not in $A'_{\tilde z}$, and 
the condition 
that $A_z' \subset Z^{-1}\{z\}$ for each $z$ precludes $(\tilde x, \tilde y)$ from being in 
$A'_z$ for any $z$ other than $\tilde z$.  
\end{proof}

\begin{proof}[Proof of Lemma \ref{lem:injectiverep}] 
Let  $Z_y$ be injective. Then any $A\subset \mathcal X\times \mathcal Y$
can be written as
$A =\{ (x,y): y\in C_{Z(x,y)} \}$ where
$C_z = \{ y \in f( Z^{-1}(z) ) : ( Z_y^{-1}(z),y) \in A \}$,
where
 $f$ is the canonical projection map from $2^{\mathcal X\times \mathcal Y} \rightarrow 2^\mathcal Y$. Conversely, suppose $Z_{y}$ is not injective
for a particular $y\in \mathcal Y$.
Then there exists $\{x_a,x_b\}\subset \mathcal X$ with $x_a \neq x_b$
but for which $Z(x_a,y) = Z(x_b,y)$.
Now let $A$ include $(x_a,y)$ but not $(x_b,y)$.
Then if $C$ is to represent $A$,  we must have $y \in C_{ Z(x_a,y)}$.
This would imply $y \in C_{ Z(x_b,y)}$,
since $Z(x_a,y) = Z(x_b,y)$. So $C$ cannot represent $A$, since
$(y,x_b)\not\in A$ by assumption.
\end{proof}

\begin{proof}[Proof of Corollary \ref{cor:accF}] 
Apply Lemma \ref{lem:injectiverep} with the roles of $\mathcal X$ and $\mathcal Y$ interchanged. 
\end{proof}

\begin{proof}[Proof of Theorem \ref{thm:ccrep}]   
Suppose $P_\theta(A) =1-\alpha$ for all $\theta$. 
For each $z\in \mathcal Z$ define $A_z= A \cap Z^{-1}\{z\}$.  
Then 
$\cup A_z = A \in \mathcal A$ and $A_z\subset Z^{-1}\{z\}$ so conditions 1 and 2 are satisfied. 
By 
regular sufficiency,
the set $H_1 = \{ z: P_z(Z^{-1}\{z\})=1\}$ 
has $P_\theta$-probability one for each $\theta$, 
and 
by completeness, the set $H_2= \{ z: P_z(A) =1-\alpha\}$ 
also has $P_\theta$-probability one for each $\theta$. 
Therefore, the 
set $H = H_1\cap H_2$ has $P_\theta$-probability one 
for each $\theta$. 
For each $z$ in $H$,  we therefore have 
$P_z(A_z) = P_z(A) = 1-\alpha$, and so condition 3 is satisfied.  
Conversely, suppose $\{A_z: z\in \mathcal Z\}$ defines a set-valued 
function that satisfies conditions 1, 2 and 3. 
Let $A = \cup_{z\in \mathcal Z} A_z$ and note that 
$A\cap Z^{-1}\{z\} = A_z$ by condition 2.  
Let $H=  
\{ z: P_z(A_z) = 1-\alpha \} \cap \{ z: P_z( Z^{-1}\{z\})=1\}$. 
By condition 3 and regular sufficiency of $Z$, 
$\nu_\theta(H)=1$ for all $\theta$,
where $\nu_\theta$ is the marginal probability measure of $Z$ under 
$P_\theta$. For all $\theta$ we therefore have
\begin{align}
P_\theta(A)    =  \int P_z(A) \, \nu_\theta(dz)  
    & =  \int_H P_z(A\cap Z^{-1}\{z\}) \, \nu_\theta(dz) \\
    & = \int_H P_z(A_z) \, \nu_\theta(dz) \\ 
    & = \int_H (1-\alpha) \, \nu_\theta(dz)  = 1-\alpha. 
\end{align}  
\end{proof}

\begin{proof}[Proof of Theorem \ref{thm:htrep}]
The proof is essentially the same as that of Theorem \ref{thm:ccrep}
and so is omitted. 
\end{proof}

\begin{proof}[Proof of Lemma \ref{lem:nplemma}]   
For sets $A$ and $\tilde A$ described in the lemma, 
\begin{equation}
  0 \leq  P(A) - P(\tilde A) = P(A\cap \tilde A^c ) - P(A^c \cap \tilde A).
\end{equation} 
Similarly, 
$R(A) - R(\tilde A) = R(A\cap \tilde A^c) - R(A^c \cap \tilde A)$. 
The first term in this  difference is 
\begin{align}
R(A\cap \tilde A^c) &= \int_{A \cap \tilde A^c} r(\omega) \, \lambda(d\omega)\\ 
  & \geq   \int_{A \cap \tilde A^c}  [p(\omega)/k] \, \lambda(d\omega) 
  = P(A\cap \tilde A^c )/k.
\end{align}  
As for the second term in the risk difference, in the case that
$P(A^c \cap \tilde A)>0$
we have 
\begin{equation} P(A^c \cap \tilde A)/k = \int_{A^c\cap \tilde A} [ p(\omega)/k] 
\, \lambda(d\omega) > \int_{A^c\cap \tilde A} r(\omega) \lambda(d\omega) =
  R(A^c\cap \tilde A) \end{equation}
  because $r(\omega) < p(\omega)/k$ on  $\tilde A$. 
Therefore, if 
$P(A^c \cap \tilde A) >0$, we have  
$R(A\cap \tilde A^c) \geq P(A\cap \tilde A^c)/k$, 
$R(A^c\cap \tilde A) < P(A^c\cap \tilde A)/k$,  and so 
\begin{align}
R(A) - R(\tilde A) & = R(A\cap \tilde A^c) - R(A^c \cap \tilde A) \\ 
   & > [ P(A\cap \tilde A^c) - P(A^c \cap \tilde A)  ]/k \geq 0. 
\end{align}
In the case that 
$P(A^c \cap \tilde A) =0$
we must also have $R(A^c\cap \tilde A) =0$ since $r(\omega) < 
    p(\omega)/k$  on $\tilde A$.  The risk difference in this case is then  
    just $R(A \cap \tilde A^c)$, which was already shown to be 
    greater than or equal to $P(A \cap \tilde A^c)$, which is
greater than or equal to $P(A \cap \tilde A^c)$, which is zero in this case. 
Therefore, 
the risk difference is positive and can only be zero 
if $P(A \Delta \tilde A) = 0$. 
\end{proof}

\begin{proof}[Proof of Theorem \ref{thm:absthm}]  
Let $A_z= A \cap Z^{-1}\{z\}$ for each $z$. 
By regular sufficiency and completeness, there exists 
a set $H_P\in \mathcal H$ such that
$\nu_\theta(H_P) = 1$ for all $\theta$ and 
$P_z(A^\pi_z) = P_z(A_z)$ for all 
$z\in H_P$.  
Let $H_R= \{ z: 
     R_z( Z^{-1}\{z\} ) = 1-R_z(Z^{-1}\{z\}^c) = 1 \}$, that is, 
the subset of $\mathcal Z$ for which $R_z$ is a 
probability measure concentrated on $Z^{-1}\{z\}$.
Now $\nu_R(H_P^c)=0$ by assumption 1, and 
$\nu_R(H_R^c) =0$ by Theorem 2 of \citet{chang_pollard_1997}, 
and so $\nu_R( (H_P \cap H_R)^c )=0$  as well. 
This implies that 
\begin{align}
R(A)   = \int R_z( A ) \, \nu_R(dz) 
 &=  \int_{H_P \cap H_R } R_z(A_z) \, \nu_R(dz)  \\
&  \geq \int_{H_P \cap H_R} R_z(A^\pi_z) \, \nu_R(dz) = 
   R(A^\pi), 
\end{align}
where the second line holds because $R_z(Z^{-1}\{z\}^c )=0$ 
for all $z \in H_R$, 
and the third line holds because for $z\in H_P\cap H_R$, 
$R_z$ is a probability measure ($z\in H_R$), 
$P_z(A_z) = P_z(A_z^\pi)$ ($z\in H_P$), and 
so by Lemma \ref{lem:nplemma} and the form of $A^\pi_z$, 
we have $R_z(A_z) \geq R_z(A_z^\pi)$.  

Condition 2 implies almost sure
uniqueness of $A^\pi$ as a minimizer of risk, 
as long as the risk is finite:
Suppose $R(A^\pi)<\infty$ and $R(A) = R(A^\pi)$. 
Then the set $\{ z: R_z(A_z) > R_z(A^\pi_z) \}$ has 
$\nu_R$-measure zero. 
This implies
that this set also has $\nu_\theta$-measure zero for every $\theta$
by condition 2, and so 
$R_z(A_z) = R_z(A^\pi_z)$
with probability one under each $\nu_\theta$. 
But by Lemma \ref{lem:nplemma}, 
$R_z(A_z) = R_z(A^\pi_z)$ is only possible for a given $z$ 
if $P_z(A_z \Delta A_z^\pi) = 0.$ Therefore, 
we must have $P_z(A_z \Delta A_z^\pi) = 0$ with $\nu_\theta$-probability 
one for each $\theta$. Therefore, equality of finite risks implies 
\begin{align}  
P_\theta(A\Delta A^\pi) &=  
\int P_z( A\Delta A^\pi ) \, \nu_\theta(dZ)   \\
&  = \int P_z( A_z\Delta A^\pi_z ) \, \nu_\theta(dZ) =0. 
\end{align}
\end{proof}

\begin{proof}[Proof of Corollary \ref{cor:cpcor}]   
Let $\gamma$ be a $\sigma$-finite Radon measure on a metric space 
$\Omega$ with Borel sets $\mathcal A$, and let $Z:(\Omega,\mathcal A)
\rightarrow  ( \mathcal Z, \mathcal H)$ be a measurable function. Theorem 1 of \citet{chang_pollard_1997} says that if $\mathcal H$ is countably generated and contains the 
singleton sets, and the image measure $\nu$
 of $\gamma$ under $\mathcal Z$  
is $\sigma$-finite, then $\gamma$ has a $(Z,\nu)$ 
disintegration. 
Corollary \ref{cor:cpcor} is a corollary of this theorem if, under the 
assumptions of the corollary, 
\begin{enumerate}
\item $R$ is a $\sigma$-finite Radon measure on 
the product $\sigma$-algebra
$\mathcal A=\mathcal F\otimes \mathcal G$, and 
\item $\mathcal H$ is countably generated and contains the 
singleton sets. 
\end{enumerate}
Item 2 holds under the assumption that $\mathcal Z$ 
is a separable metric space and $\mathcal H$ the Borel $\sigma$-algebra.
Regarding item 1, 
since $P^X_\pi$ and $\mu$ are both $\sigma$-finite, so is their 
product measure $R$. Also,  
since $\mathcal X$ is complete and separable then 
any probability measure on $(\mathcal X, \mathcal F)$, including $P_\pi^X$, 
is a Radon measure \citep[Theorem 1.4]{billingsley_1968}. 
By assumption, $\mu$ is also a $\sigma$-finite Radon measure. 
Now let $\mathcal B$ be the Borel sets of $\mathcal X\times \mathcal Y$ 
under the product topology, and recall that $\mathcal A \subset \mathcal B$
in general. 
By 
\citet[Theorem 9.6.3]{lang_1993}, 
assuming $\mathcal X$ and $\mathcal Y$ are $\sigma$-compact, 
there exists a unique Radon measure $\tilde R$ on 
$\mathcal B$ that matches $R$ on $\mathcal A$. 
But since  $\mathcal X$ and $\mathcal Y$ are separable spaces, 
$\mathcal A = \mathcal B$  \citep[Theorem 6.4.2]{bogachev_2007}.  
Therefore, $R = \tilde R$ and so 
$R$ is a Radon measure. 
\end{proof}

\begin{proof}[Proof of Lemma \ref{lem:excomplete} ]
Define $N_\delta = \{ z: h(z) < -\delta \}$ for $\delta> 0$. We will first show that
$\nu_\theta(N_\delta) = 0$ for all $\theta$ by contradiction.
Suppose $N_\delta$ is non-null. Then by the assumption, for any $\epsilon>0$
there exists  a $\theta_\epsilon$
such that $\nu_{\theta_\epsilon}(N_\delta) > 1-\epsilon$. This implies
\begin{align}
\int h(z) \, \nu_{\theta_\epsilon}(dz) &= \int_{N_\delta} h(z)  \, \nu_{\theta_\epsilon}(dz) +
  \int_{N_\delta^c} h(z)  \, \nu_{\theta_\epsilon}(dz) \\
  &  < -\delta (1-\epsilon) + \lVert h\rVert_{\infty} \epsilon
\end{align}
which is less than zero for $\epsilon< \delta/(\delta+ \lVert h\rVert_\infty)$.
Therefore, if $\int h(z) \, \nu_\theta(dz) \geq 0$ for all $\theta$,
then $\nu_\theta(N_\delta) =0$ for all $\delta>0$ and  $\theta$.
Now note that $\{z: h(z)<0\} = \cup_{k=1}^\infty N_{1/k}$, and so
\begin{align}
\nu_\theta(\{z: h(z)<0\})   = \nu_\theta( \cup_{k=1}^\infty N_{1/k}) 
                  = \lim_{k\rightarrow \infty} \nu_\theta(N_{1/k}) = 0
\end{align}
because
$N_{1/k} \subset  N_{1/(k+1)}$ and
 $\nu_\theta(N_{1/k})=0$ for all positive integers $k$.
\end{proof}

\begin{proof}[Proof of Theorem \ref{thm:boptlevel}]
Let $A\in \mathcal A$ satisfy $P_\theta(A)\geq P_\theta(A^\pi)$ for all $\theta\in \Theta$. Then $\int (P_z(A) - P_z(A^\pi))\, \nu_\theta(dz)$ is non-negative
for
all $\theta$, and so by Lemma \ref{lem:excomplete}, $P_z(A) \geq P_z(A^\pi)$ 
almost surely $\nu_\theta$ for each $\theta$. 
By Lemma \ref{lem:nplemma} we then also have 
$R_z(A) \geq R_z(A^\pi)$ almost surely $\nu_\theta$ for each $\theta$, 
which implies  $R_z(A) \geq R_z(A^\pi)$ almost everywhere $\nu_R$ by 
Assumption 1 of Theorem \ref{thm:absthm}. 
Finally, this gives 
\begin{equation}
R(A)-R(A^\pi)=\int( R_z(A) - R_z(A^\pi) ) \, \nu_R(dz)  \geq 0. 
\end{equation}
\end{proof}

\begin{proof}[Proof of Theorem \ref{thm:rnderiv}]  
The density of $P_\theta$ with respect to $R$ is 
$h_\theta(Z(x,y)) g(x,y)/p_\pi^X(x)$, and so 
\begin{align}
P_\theta(A)&  = \int_A h_\theta(Z(x,y)) [g(x,y)/p_\pi^X(x)] \, R(dx \, dy) 
  \nonumber \\ 
&= \int \left ( \int_A h_\theta(Z(x,y)) [g(x,y)/p_\pi^X(x)]  \, R_z(dx \, dy)  
\right ) \, \nu_R(dz)  \nonumber  \\
&= \int h_\theta(z) 
  \left ( \int_A  [g(x,y)/p_\pi^X(x)] \, R_z(dx\, dy)  \right ) \, 
   \nu_R(dz) \nonumber  
\equiv  \int h_\theta(z) Q_z(A)  \, \nu_R(dz),  \label{eqn:PthetaQ}
\end{align}
where we have defined for each $z$ the measure $Q_z$ on $(\mathcal X\times \mathcal Y, \mathcal F\otimes \mathcal G)$ such that
$[dQ_z/dR_z](x,y) = g(x,y)/p_\pi^X(x)$.  
By Fubini's theorem, 
$Q_z( \mathcal X\times \mathcal Y)$ is finite for $\nu_R$-almost all $z$
since $P_\theta(\mathcal X\times \mathcal Y)$ is finite. 
Additionally, since 
$R_z$ concentrates on $Z^{-1}\{z\}$ 
for $\nu_R$-almost all $z$ and $Q_z\ll R_z$ for each $z$, we have that 
 $Q_z$ concentrates on $Z^{-1}\{z\}$ 
for $\nu_R$-almost all $z$ as well. 

The measures $\{ Q_z : z\in \mathcal Z\}$ can be related to 
a conditional distribution $\{P_z:z\in \mathcal Z\}$ 
for the family $\{ P_\theta:\theta\in \Theta \}$, which 
then gives an expression for $dP_z/dR_z$.  
First, note that 
$\{ h_\theta(z)Q_z : z\in \mathcal Z\}$ gives a 
$\{ Z, \nu_R\}$-disintegration of $P_\theta$.  
For $H\in \mathcal H$,
\begin{align}
\nu_\theta(H) \equiv P_\theta(Z^{-1} H) & = 
 \int h_\theta(z) Q_z(Z^{-1} H)  \, \nu_R(dz) \\  
&= \int h_\theta(z) \left(  \int 1( Z(x,y) \in H) \, Q_z(dx\, dy) \right )  
   \, \nu_R(dz) \\ 
&= \int_H h_\theta(z) \int Q_z(dx \, dy) \, \nu_R(dz)    \\
&= \int_H h_\theta(z) Q_z(\mathcal X\times \mathcal Y) \, \nu_R(dz),  
\end{align}
and so 
$[d \nu_\theta/ d\nu_R](z) = h_\theta(z) Q_z(\mathcal X\times \mathcal Y)$. 
Therefore, continuing from (\ref{eqn:PthetaQ}) we have
\begin{align} 
P_\theta(A) & = \int h_\theta(z) Q_z(A) \, \nu_R(dz) \nonumber \\  &= 
       \int [Q_z(A)/Q_z(\mathcal X\times \mathcal Y)] h_\theta(z) 
        Q_z(\mathcal X\times \mathcal Y)  \, \nu_R(dz) \nonumber  \\
 & = \int [Q_z(A)/Q_z(\mathcal X\times \mathcal Y)]  \, \nu_\theta(dz) 
  \equiv   \int P_z(A) \, \nu_\theta(dz)  
\label{eqn:rcp} 
\end{align}
so that 
$P_z$ is a probability measure
on $(\mathcal X\times \mathcal Y, \mathcal A)$. 
We now show that $\{P_z:z\in \mathcal Z\}$ is a regular conditional 
probability distribution of $P_\theta$ given $Z$ for each $\theta$. 
First, 
$P_z(Z^{-1}\{z\}) =1$ 
for $\nu_R$-almost all $z$ and so also for $\nu_\theta$-almost all $z$ 
for each $\theta$, 
since $P_\theta\ll R$ for each $\theta$.  
Additionally, the map $z \rightarrow P_z(A)$ is measurable 
for each $A\in \mathcal A$. To see this, recall that $z\rightarrow R_z(A)$ 
is measurable by assumption, and so $z\rightarrow \int f(x,y) \, R_z(dx dy)$ 
is measurable for any simple function $f$. Letting
$f_k(x,y) \uparrow g(x,y)/p_\pi^X(x)$ as $k\rightarrow \infty$ 
we have $Q_z(A) = \lim_{k\rightarrow \infty} \int f_k(x,y) \, R_z(dx dy)$, 
so $z\mapsto Q_z(A)$ is a limit of measurable functions and so is measurable. 
Thus, $z\mapsto P_z(A)= Q_z(A)/Q_z(\mathcal X\times \mathcal Y)$ 
is measurable as well. Finally, 
\begin{equation} dP_z/dR_z = [ Q_z(\mathcal X\times \mathcal Y)  ]^{-1} dQ_z/dR_z =
[ Q_z(\mathcal X\times \mathcal Y)  ]^{-1} g(x,y)/p_\pi^X(x) .  \end{equation}
\end{proof}

\begin{proof}[Proof of Corollary \ref{cor:nfrep}] 
The corollary results from plugging in the form of $dP_z/dR_z$ from 
Theorem \ref{thm:rnderiv} into the expression for the optimal 
set in Theorem \ref{thm:absthm}.
\end{proof}

\begin{proof}[Proof of Theorem \ref{thm:nppred}]  
By Corollary \ref{cor:cpcor}, item 3 will be true if items 1 and 
2 are true. 
Regarding item 1, 
the risk measure $R= P_\pi^X\times \mu$  
is a $\sigma$-finite Radon measure by the same arguments as in 
the proof of 
Corollary \ref{cor:cpcor}. Regarding item 2, 
first we show that the sufficient statistic $Z$ is a 
measurable map from $(\mathcal Y^{n+1},\mathcal G^{n+1})$ to 
$(\mathcal Z,\mathcal H)$, 
where $\mathcal Z$ is the 
subset of $\mathcal P$ consisting of measures with $n+1$ 
or fewer support points and $\mathcal H$ is the $\sigma$-algebra 
generated by the weak topology on $\mathcal Z$. 
Because $\mathcal Y$ is separable, 
$\mathcal Z$ 
with this topology can be separably metrized
 \citep[Theorem 6.2]{parthasarathy_2005}.  
In this case, 
$\mathcal H$ is the same as the $\sigma$-algebra generated by 
sets of the form $\{ P\in \mathcal P : P(G)<c\}$ for 
$c\in [0,1]$,  $G\in \mathcal G$ \citep{karr_1991}. 
Let $H$ be such a set. Then 
$Z^{-1}H = \{ (y_1,\ldots,y_{n+1})\in \mathcal Y^{n+1} : \sum 1(y_i \in G)/(n+1) < c \}$, which 
is a measurable subset of $\mathcal G^{n+1}$. 

Let $\nu_R$ be the image measure of $R$ under $Z$, defined 
by $\nu_R(H) = R(Z^{-1} H)$ for $H\in \mathcal H$. 
For any $G\in \mathcal G$, define 
$H_G =\{ P\in \mathcal Z : P(G) =1 \}$. 
Then $Z^{-1} H_G =\{ (y_1,\ldots, y_{n+1})\in \mathcal Y^{n+1} : \{ y_1,\ldots, y_{n+1} \} 
 \subset G \} =  G^n$.  
Therefore, $\nu_R( H_G ) = P_\pi^X (G^n) \times \mu(G)$, 
which is less than infinity if $\mu(G)<\infty$. 
Now by $\sigma$-finiteness of $\mu$, 
there exists a countable set $\{ G_k: k\in \mathbb N\} \subset \mathcal G$ 
such that $\cup_1^\infty G_k = \mathcal Y$,
$G_k \subset G_{k+1}$ and 
$\mu(G_k) <\infty$ for each $k\in \mathbb N$. 
Therefore, $\nu_{R}(H_{G_k}) <\infty$ for each $k\in \mathbb N$ as well. 
We now show $\cup_k H_{G_k} = \mathcal Z$: 
Because $\cup_k G_k = \mathcal Y$, 
for each $y\in \mathcal Y$ there is some $k_y$ such that $y\in G_l$
for all $l\geq k_y$. 
Let $z\in \mathcal Z$, and let $y_1,\ldots, y_{m}$ be the support points 
of $z$ ($m$ could be $n+1$ or smaller).  
Then $z\in H_{G_k^*}$ where 
$k^* = \max \{ k_{y_1},\ldots, k_{y_m} \}$. 
Thus $\cup_k H_{G_k} = \mathcal Z$, and so 
$\nu_R$ is $\sigma$-finite. The assumptions of 
Corollary \ref{cor:cpcor} are met and so 
$R$ has a $(Z,\nu_R)$-disintegration.  

Complete sufficiency of $Z$ under assumption A1 follows from
\citet{bell_blackwell_breiman_1960}. 
The remaining conditions for $A^\pi$ to be risk-optimal for its coverage 
function 
are conditions 1 and 2 
of Theorem \ref{thm:absthm} and
condition 1 of 
Theorem \ref{thm:rnderiv}.  Assumption A2 immediately 
implies the last two of these. 
The remaining condition is that if 
$P^{n+1}(Z^{-1} H) =0$ for each 
$P\in \mathcal P$, 
we must have $\nu_R(H)=0$. 
Assumption A2 implies something stronger, that 
for $E\in \mathcal G^{n+1}$, 
$P^{n+1}(E) = 0$ for all $P\in \mathcal P$ implies 
$R(E) = 0$, or conversely, that 
$R(E)>0$ implies there exists a $P\in \mathcal P$ such that 
$P^{n+1}(E) > 0$. To see this, recall that $\sigma$-finiteness of $\mu$ 
implies the existence of a strictly positive function 
$f$ on $\mathcal Y$ for which $\int f(y) \mu(dy) = 1$. Let 
$P$ be the measure on $(\mathcal Y,\mathcal G)$ 
with density $f$ with respect to $\mu$, and 
$P^{n+1}$ its $n+1$-fold product measure. 
Then $P^{n+1}(E) >0$ if $\mu^{n+1}(E)>0$ or equivalently, if $R(E)>0$. 
\end{proof}

\bibliography{fpredict}

\end{document}